\documentclass[11pt,reqno]{amsart}
\usepackage[utf8]{inputenc}
\usepackage{amsmath,amssymb,amsthm,mathrsfs, color}
\usepackage{a4wide}
\usepackage{epsfig}
\usepackage{caption}
\usepackage{subcaption} 
\usepackage{graphics}
\usepackage{epsfig}

\usepackage[round]{natbib}

\usepackage{comment} 

\usepackage[colorlinks=true]{hyperref}
\hypersetup{%
    ,urlcolor=blue
    ,citecolor=red
    ,linkcolor=blue
    }

\newtheorem{theorem}{Theorem}[section]
\newtheorem{prop}[theorem]{Proposition}
\newtheorem{lemma}[theorem]{Lemma}
\newtheorem{corollary}[theorem]{Corollary}
\newtheorem{definition}[theorem]{Definition}
\newtheorem{remark}[theorem]{Remark}
\newtheorem{assumption}[theorem]{Assumption}
\newtheorem{example}[theorem]{Example}

\newcommand{\cM}{\mathcal{M}}

\newcommand{\R}{{\mathbb R}}
\newcommand{\N}{{\mathbb N}}

\newcommand{\cF}{{\mathcal F}}

\renewcommand{\P}{{\mathbb P}}
\newcommand{\F}{\mathbb{F}}
\newcommand{\E}{{\mathbb E}}

\newcommand{\hatX}{\widehat{X}}

\DeclareFontFamily{U}{mathx}{\hyphenchar\font45}
\DeclareFontShape{U}{mathx}{m}{n}{<-> mathx10}{}
\DeclareSymbolFont{mathx}{U}{mathx}{m}{n}
\DeclareMathAccent{\widebar}{0}{mathx}{"73}

\def\EKG{\mathscr{E}_K^{\Gamma}}

\newcommand{\ind}{\mathbf{1}}

\newcommand{\ud}{\mathrm{d}}

\makeatletter
 \@addtoreset{equation}{section}
 \makeatother
 
\usepackage{appendix}
\usepackage{hyperref}

\usepackage{todonotes}

\usepackage{mathtools}
\mathtoolsset{showonlyrefs}

\title[Weak solutions of SVE in convex domains with general kernels]{Weak solutions of Stochastic Volterra Equations in convex domains with general kernels}

\author[E. Abi Jaber]{Eduardo Abi Jaber}
\address{CMAP, Ecole Polytechnique, Palaiseau, France.}
\email{eduardo.abi-jaber@polytechnique.edu}

\author[A. Alfonsi]{Aur\'elien Alfonsi}
\address{CERMICS, Ecole des Ponts, Marne-la-Vall\'ee, France. MathRisk, Inria, Paris,
	France.}
\email{aurelien.alfonsi@enpc.fr}

\author[G. Szulda]{Guillaume Szulda}
\address{CERMICS, Ecole des Ponts, Marne-la-Vall\'ee, France.}
\email{guillaume.szulda@enpc.fr}

\thanks{EAJ is  grateful for the financial support from the Chaires FiME-FDD, Financial Risks and Deep Finance \& Statistics at Ecole Polytechnique. AA and GS benefited from the support of the Chaire Financial Risks}
\keywords{ Stochastic Volterra equations; Weak solution; Stochastic invariance; Double kernels}
\subjclass[2020]{60H20, 45D05, 60B10}
\date{\today}

\begin{document}

\begin{abstract}
We establish new weak existence results for $d$-dimensional Stochastic Volterra Equations (SVEs) with continuous coefficients and possibly singular one-dimensional non-convolution kernels. These results are obtained by introducing an approximation scheme and showing its convergence. A particular emphasis is made on the stochastic invariance of the solution in a closed convex set. To do so, we extend the notion of kernels that preserve nonnegativity introduced in \cite{Alfonsi23} to non-convolution kernels and show that, under suitable stochastic invariance property of a closed convex set by the corresponding Stochastic Differential Equation, there exists a weak solution of the SVE that stays in this convex set. We present a family of non-convolution kernels that satisfy our assumptions, including a non-convolution extension of the well-known fractional kernel. We apply our results to SVEs with square-root diffusion coefficients and non-convolution kernels, for which we prove the weak existence and uniqueness of a solution that stays within the nonnegative orthant. We derive a representation of the Laplace transform in terms of a non-convolution Riccati equation, for which we establish an existence result.
\end{abstract}

\maketitle

\section{Introduction}\label{Sec_intro}

The aim of the paper is to study the weak existence  of continuous  solutions to the following $d$-dimensional Stochastic Volterra Equation (SVE):
\begin{equation}\label{eq:SVE}
X_t = X_0 + \int_0^t \Gamma(t,s)\,b(X_s)\,\mathrm{d}s + \int_0^t \Gamma(t,s)\,\sigma(X_s)\,\mathrm{d}B_s,
\end{equation}
where $B$ is a $d$-dimensional Brownian motion, and $X_0 \in \mathbb{R}^d$.  Here, $b : \mathbb{R}^d \to \mathbb{R}^d$ represents the drift coefficient, $\sigma : \mathbb{R}^d \to \mathcal{M}_d(\mathbb{R})$ is the diffusion coefficient, and $\Gamma:\R_+^2 \to \R$ is a locally square-integrable kernel. In addition, for a closed convex subset $\mathscr{C} \subset \mathbb{R}^d$ and $X_0\in \mathscr{C}$, we are particularly interested in conditions on $(\Gamma, b,\sigma)$ that ensure the existence of a weak solution $X$ that remains in the set $\mathscr{C}$ for all times, a problem known as stochastic invariance. Our general framework covers both convolution and non-convolution  (possibly unbounded) kernels~$\Gamma$,  the latter introducing several challenges for the analysis of weak existence and stochastic invariance.

First, even in the unconstrained case, i.e.~when $\mathscr{C} = \mathbb{R}^d$, the existence of weak solutions  for continuous $(b,\sigma)$ is not established in the literature for non-convolution kernels $\Gamma$ that are unbounded on the diagonal, that is, exhibiting a singularity at $\Gamma(t,t)$. \cite{Zhang10} has obtained strong existence results with singular kernels but under globally Lipschitz assumptions on $(b,\sigma)$. The closest existing results to our work concern convolution kernels of the form $\Gamma(t,s) = \mathbf{1}_{s < t} K(t-s)$, where $K : \mathbb{R}_+ \to \mathbb{R}$ is locally square-integrable. These works rely crucially on the convolution structure to control the weak convergence and stability of the stochastic convolution $\int_0^t K(t-s) \sigma(X_s)dB_s$ in the non-semimartingale setting, either via the resolvent of the first kind \cite*[Theorem 3.4]{AJLP19}, or through integration and Fubini-type arguments that simplify the analysis \cite*[Theorem 1.2]{AJCLP21}. In contrast, weak existence results for non-convolution kernels $\Gamma$ are far less developed. To bypass the difficulty, existing approaches often impose strong regularity conditions on $\Gamma$, ruling out singular kernels and effectively reducing the problem to a semimartingale framework; see for instance \cite*[Theorem 3.3]{PS23Weak}.

In Theorem~\ref{thm_weak_Rd}, we provide the first weak existence results for SVEs with non-Lipschitz coefficients and non-convolution Volterra kernels allowing for singularities on the diagonal.

Second, we turn to the constrained case, and study the stochastic invariance problem for general closed convex subsets $\mathscr{C} \subset \mathbb{R}^d$. Stochastic invariance for Volterra equations has recently emerged as a central question, motivated in particular by the nonnegative Volterra square-root process \citep{dawson1994super, mytnik2015uniqueness, jaisson2016rough} and its applications in mathematical finance, see \cite*{EER19}. Existing results focus on convolution kernels of the form $\Gamma(t,s) = \mathbf{1}_{s < t} K(t - s)$. In this setting, \cite*{AJLP19} established weak existence for solutions to \eqref{eq:SVE} in the nonnegative orthant $\mathbb{R}^d_+$ under suitable conditions on $b$ and $\sigma$, assuming that the kernel $K$ is nonincreasing, continuous, and has a nonnegative and nonincreasing resolvent of the first kind. \cite{AJ21,CT19} further extended these results to affine SVEs with jumps under similar kernel conditions. Additionally, ~\cite*{abi2024polynomial} constructed polynomial Volterra processes constrained to the unit ball. More recently, \cite{Alfonsi23} established sufficient conditions for stochastic invariance of closed convex sets $\mathscr{C}$ under SVEs with Lipschitz coefficients and convolution kernels. The novel idea is to consider  kernels that preserve nonnegativity, in the sense that,  if a discrete convolution remains nonnegative at specific discretization points, then the convolution remains nonnegative at all times. 
Extensions to the case of SVEs with jumps in the nonnegative half-line $\mathscr{C} = \mathbb{R}_+$ were further developed by \cite{AS24}.

Our main result, Theorem~\ref{thm_maininvariance}, unifies and extends the theory of stochastic invariance for Stochastic Volterra Equations of the form \eqref{eq:SVE} by allowing for non-Lipschitz coefficients and non-convolution kernels and general closed convex subsets, thereby addressing significant gaps in the literature. 

This extends existing invariance results and covers new classes of constrained SVEs, including examples that were previously out of reach:
\begin{itemize} \item Matrix-valued solutions constrained to the cone of symmetric positive semi-definite matrices, as in Wishart-type Volterra processes. To the best of our knowledge, none of the existing results apply to this setting — even in the convolution case, the conditions in \cite{Alfonsi23} require Lipschitz coefficients, which are not satisfied here. We also note that unbounded kernels seem to be  ruled out in this context, which motivated different constructions by \cite{abi2022laplace, CT20}. \item Non-convolution extensions of affine and polynomial Volterra processes, constrained to state spaces such as the nonnegative orthant or the unit ball. \end{itemize}

Since the theory of resolvents of the first kind is so far limited to convolution kernels, we adopt a different approach to handle non-convolution kernels. We extend the notion of kernels that preserve nonnegativity introduced in \cite{Alfonsi23} for the convolution case to more general non-convolution kernels in Definition~\ref{def:positivity}. In particular, we show how such kernels can be constructed from a suitable class of completely monotone double kernels, which generalize the classical class of completely monotone kernels known in the convolution setting, see Definition~\ref{def:completely_monotone_double} and Theorem~\ref{thm:completely_monotone_double}. Furthermore, we provide a dedicated analysis of kernels that preserve nonnegativity in Section~\ref{Sec_kernels}, a property of independent interest.

Finally, in Section~\ref{S:squareroot}, we establish weak existence and uniqueness for an affine Volterra square-root process with a non-convolution kernel. Weak uniqueness follows from an explicit representation of the Laplace transform in terms of a Volterra Riccati equation, for which we provide a corresponding existence result, see Theorem~\ref{T:VolSqrt}. This extends the convolution-based results of \cite*[Section 6]{AJLP19} to the non-convolution setting and fills the existence gaps for both the stochastic equation and the associated Volterra Riccati equation considered in \cite*{AKO22}.
In particular, we obtain as an important application the weak existence and uniqueness, and the stochastic invariance for the SVE \begin{equation}\label{eq:SVE_fractional_R+} 
	X_t = X_0 + \int_0^t {G\biggl(\int_s^t {h(u)\,\ud u}\biggr)\Bigl( (\theta - \lambda X_s)\,\ud s + \sigma\sqrt{X_s}\,\ud B_s\Bigr)},
\end{equation}
with $X_0\geq0$, $\theta \geq 0$, $\lambda \in \mathbb{R}$ and $\sigma > 0$, when $G:\R_+^*\to\R$ is a convolution kernel satisfying certain integrability and regularity conditions and $h:\R_+^*\to\R_+^*$ is a locally integrable function, locally bounded away from zero and such that $t \mapsto \int_0^t h(u)\ud u$ is locally Hölder continuous. In particular, when setting $G(t) = \frac{t^{\alpha-1}}{\Gamma_e(\alpha)}$ where $\alpha\in(\frac{1}{2},1]$ and $\Gamma_e$ denotes the Euler Gamma function, we obtain a non-convolution extension of the well-known fractional kernel.

The paper is organized as follows. Section~\ref{Sec_Mainresult} collects our main results on weak existence and stochastic invariance for SVEs. Section~\ref{S:ecamples} provides explicit examples of non-convolution kernels and coefficients illustrating our framework. In Section~\ref{S:squareroot}, we study weak existence and uniqueness for Volterra square-root processes with non-convolution kernels. Sections~\ref{S:proofs_main_existence} and~\ref{S:proofsquareroot} contain the proofs of our main results. Section~\ref{Sec_kernels} then provides a dedicated analysis of non-convolution kernels that preserve nonnegativity. Finally, the appendix recalls some background material on stochastic invariance for stochastic differential equations and provides a technical approximation lemma for kernels.

\textbf{Notation.} We define the following sets
\begin{align*}
	&\Delta=\{(t,s)\in \R_+: s\le t \}, \ \mathring{\Delta}=\{(t,s)\in \R_+: 0<s< t \},\\
	&\Delta_T=\{(t,s)\in \R_+: s\le t \le T \}, \ \mathring{\Delta}_T=\{(t,s)\in \R_+: 0\,<\,s<t<T \}, \text{ for } T\in \R_+^*.
\end{align*}
$\cM_d(\R)$ denotes the set of real square matrices of size $d$ endowed with the Frobenius norm $|\cdot|$. 

\section{Main existence results}\label{Sec_Mainresult}
In this section, we present our main results  - Theorem~\ref{thm_maininvariance} and Corollary~\ref{cor_strong_domain}
-  on existence of $\mathscr{C}$-valued solutions to the stochastic Volterra equation \eqref{eq:SVE}. We start by establishing  a weak existence result for unconstrained continuous solutions in $\mathbb R^d$ to the   SVE \eqref{eq:SVE}, in Theorem~\ref{thm_weak_Rd}. Then, we detail a  domain-preserving approximation scheme in Section~\ref{S:scheme} which is the key ingredient behind both Theorems~\ref{thm_weak_Rd} and \ref{thm_maininvariance}.

 By a \emph{weak continuous solution} to the SVE \eqref{eq:SVE}, we mean the existence of a filtered probability space $(\Omega,\cF,\F:=(\cF_t)_{t\geq0},\P)$, satisfying the usual conditions and supporting an $\F$-Brownian motion $B=(B_t)_{t\geq0}$ as well as an $\F$-adapted pathwise continuous process $X=(X_t)_{t\geq0}$, such that Equation~\eqref{eq:SVE} holds almost surely for all $t\geq0$.

We introduce the following mild  continuity and linear growth conditions on the coefficients $b:\R^d \to \R^d$, $\sigma:\R^d \to \mathcal M_d(\R)$ and square-integrability conditions on the  Volterra kernel $\Gamma:\mathring{\Delta} \to \R$.
\begin{assumption}\label{ass:conditions1_coefficients} The coefficients $b$ and $\sigma$ are continuous and 
	there exists $C_{LG}\in\R_+$ such that
	\begin{equation*}
	|b(x)| + |\sigma(x)| \leq C_{LG}\,(1 + |x|), \quad x\in\R^d.
	\end{equation*}
\end{assumption}
\begin{assumption}\label{ass:conditions1_kernel}
	For every $T\in(0,+\infty)$, there exist $\eta>0,\gamma\in(0,1/2]$ such that
	\begin{equation*}
	\int_s^t {\Gamma(t,u)^2\,\ud u} + \int_0^s {(\Gamma(t,u) - \Gamma(s,u))^2\,\ud u} \leq \eta\,(t-s)^{2\gamma},\quad (t,s)\in\Delta_T.
	\end{equation*}
\end{assumption}

Assumption~\ref{ass:conditions1_kernel} accommodates kernels $\Gamma$ that may exhibit singularities on the diagonal. Furthermore, for such kernels the solution $X$ is no longer expected to be a semimartingale, which poses significant challenges in constructing weak solutions via approximation methods.
We now state an important lemma to approximate the kernels. Its proof is postponed to Appendix~\ref{A:proofkernel}. 
\begin{lemma}\label{lem_approx_kernel}
    Let $\Gamma$ satisfy Assumption~\ref{ass:conditions1_kernel}. Then, for any $T>0$, there exists a sequence of kernels $\Gamma_M:\Delta_T \to \R$ such that $\Gamma_M$ is continuous, 
     $$\int_0^t(\Gamma(t,s)-\Gamma_M(t,s))^2 \ud s \to_{M\to \infty} 0, \quad t \in[ 0, T],$$
     and 
     $$\int_s^t {\Gamma_M(t,u)^2\,\ud u} + \int_0^s {(\Gamma_M(t,u) - \Gamma_M(s,u))^2\,\ud u} \leq 2 \eta\,|t-s|^{2\gamma}, \quad (t,s)\in\Delta_T, \quad M \in \mathbb N.$$
     Besides, if there exists $\varepsilon>0$ such that either $\Gamma(t,s)\ge 0$ or $\Gamma(t,s)\le 0$ for all $(t,s) \in \Delta$ with $s\ge t-\varepsilon$, then we may choose $\Gamma_M$ such that $\Gamma_M(s,s)\not=0$ $ds$-a.e.
\end{lemma}

An important example of non-convolution kernel satisfying Assumption \ref{ass:conditions1_kernel} is given by the following natural generalization of convolution kernels.
\begin{example}\label{ex:fractional_double} 
Consider the non-convolution kernel
\begin{align*}
\Gamma(t,s) = G\left(\int_s^t {h(u)\,\ud u}\right),
\end{align*}
where $G:\R_+^*\to\R$ is a locally square-integrable convolution kernel {. The kernel $\Gamma$ can be interpreted as a time changed convolution kernel.}  Assume  that for every $T\in(0,+\infty)$, there exist $\eta>0$, $\gamma\in(0,1]$ such that
\begin{equation}\label{eq:conditions_convolution}
	\int_0^{\delta} {G(x)^2\,\ud x} + \int_0^T {(G(x + \delta) - G(x))^2\,\ud x} \leq \eta\,\delta^{2\gamma}, \quad \delta\in(0,T),
\end{equation}
and $h:\R_+^*\to\R_+^*$ is a locally integrable function such that for every $T\in(0,+\infty)$, there exist $\lambda>0$, $C > 0$ and $\beta \in (0,1]$ such that
$$ h(t)\geq\lambda \quad  \text{and} \quad \int_s^t {h(u)\,\ud u}  \leq  C(t-s)^{\beta}, \quad (t,s)\in\mathring{\Delta}_T. $$ 
Then, $\Gamma$ satisfies Assumption \ref{ass:conditions1_kernel} as follows. More specifically,  we have for all $(t,s)\in\Delta_T$, $$\int_s^t {\Gamma(t,u)^2\,\ud u} \leq \frac{1}{\lambda}\int_0^{\int_s^t {h(u)\ud u}} {G(u)^2\,\ud u} \leq \frac{\eta}{\lambda} \left(\int_s^t {h(u)\,\ud u} \right)^{2\gamma} \leq \frac{\eta}{\lambda}\,C^{2\gamma}\,(t-s)^{2\beta\gamma}$$ and 
\begin{align*}
\int_0^s {(\Gamma(t,u) - \Gamma(s,u))^2 \,\ud u}
&\leq \frac{1}{\lambda}\int_0^s {\left( G\left(\int_u^t {h(v)\,\ud v}\right) - G\left(\int_u^s {h(v)\,\ud v}\right) \right)^2 h(u)\,\ud u} \\
&\leq \frac{1}{\lambda} \int_0^{\int_0^T h(u)\, \ud u} {\left( G\left(\int_s^t {h(v)\,\ud v} + x\right) - G(x) \right)^2\ud x} \\
&\leq \frac{\eta}{\lambda} \left(\int_s^t {h(u)\,\ud u}\right)^{2\gamma} \leq \frac{\eta}{\lambda}\,C^{2\gamma}\,(t-s)^{2\beta\gamma}.
\end{align*}
Examples of convolution kernel $G$ include those of \cite[Example 2.3]{AJLP19} as \eqref{eq:conditions_convolution} corresponds to \cite[condition (2.5)]{AJLP19}. An important example is given by the fractional kernel $G(t) = \frac{t^{\alpha-1}}{\Gamma_e(\alpha)}$ where $\Gamma_e$ is the Euler Gamma function and $\alpha\in(\frac{1}{2},1]$. Interesting examples of function $h$ are given by $h(u) = e^u$ for $u\geq0$; $h(u) = u^{\beta} + C$ for $u\geq0$, $\beta\geq0$, $C>0$; $h(u) = u^{\beta-1}$ for $u>0$, $\beta\in(0,1]$.
\end{example}

To the best of our knowledge, Theorem~\ref{thm_weak_Rd} is the first weak existence results for SVEs with non-Lipschitz coefficients and  with non-convolution Volterra kernels with possible singularities on the diagonal.  

\begin{theorem}\label{thm_weak_Rd}
	  Let $\Gamma$ satisfy Assumptions~\ref{ass:conditions1_kernel} and   $b,\sigma$ satisfy Assumption~\ref{ass:conditions1_coefficients}.  Then, there exists a continuous weak solution to the SVE \eqref{eq:SVE} for any $X_0 \in \mathbb R^d$. {In addition, for any $T>0$, the paths of $X$ on $[0,T]$ are Hölder continuous of any order less than $\gamma$, where $\gamma$ is the constant associated with $\Gamma$ and $T$ in Assumption \ref{ass:conditions1_kernel}. }
    \end{theorem} 

\begin{proof}
    The proof is given in Section~\ref{S:proofs_main_existence}.
\end{proof}

The proof of Theorem~\ref{thm_weak_Rd} is achieved through an approximation argument using a scheme similar to the one detailed in the next section. Moreover, the scheme below allows the  construction of  
$\mathscr{ C}$-valued solutions to the SVE \eqref{eq:SVE} which will be detailed next.

\subsection{A domain-preserving approximation scheme for continuous kernels}\label{S:scheme}
For a closed convex subset $\mathscr{C} \subset \mathbb{R}^d$. Our main aim is to construct a $\mathscr{C}$-valued weak solution to the SVE \eqref{eq:SVE} starting from $X_0 \in \mathscr{C} $.

For this, we introduce an approximation scheme for the SVE~\eqref{eq:SVE} that will help us identify the good conditions on the coefficients $(b,\sigma)$ and the kernel $\Gamma$ to construct a $\mathscr{C} $-valued solution $X$ starting from any $X_0 \in \mathscr{C} $. This scheme is inspired by   the one proposed by \cite[Section 3]{Alfonsi23} for Lipschitz coefficients and convolution kernels. 
 {The idea of this approximation scheme is to split the dynamics in two steps: the integration with respect to the kernel on the one hand, and the "classical SDE" integration on the other hand. Then, we will put sufficient conditions to get that the scheme remains within the domain at each step. }

We set $T\in(0,+\infty)$, $N\in\N^*$ and $t_k := k\,T/N$ for each $k\in\{0,\ldots,N\}$, and assume for now that the kernel $\Gamma:\Delta_T \to \R_+$ is continuous,   {so that} $0<\Gamma(t,t)<\infty$ for all $t\in[0,T]$. We construct two c\`adl\`ag processes: an approximation scheme $\hatX^N = (\hatX_t^N)_{t\in[0,T]}$ and an auxiliary process $\xi^N=(\xi_t^N)_{t\in[0,T]}$. 

\begin{enumerate}
	\item[\underline{$k=0$:}] We define $(\hatX_t^N)_{t\in[t_0,t_1)}$ as $\hatX_t^N := X_0$ for $t\in[t_0, t_1)$ and $(\xi^N_t,B^N_t)_{t\in[0,t_1)}$ (see e.g.~\cite[Theorem IV.2.4]{IW89}) as a continuous solution of
	\begin{equation}\label{eq:SDE_initial}
	\xi_t^N = \hatX_{t_1-}^N + \int_{t_0}^t {\Gamma(t_1, t_1)\,\bigl(b(\xi_{s}^N)\,\ud s + \sigma(\xi_{s}^N)\,\ud B^N_s\bigr)}, \ t\in[t_0, t_1)
	\end{equation}
	and we note  $(\Omega^N,\cF^N,(\cF^N_t)_{t\in [0,t_1] },\P^N)$ the filtered probability space on which it is defined. 
	\item[\underline{$k=1$:}] We then define $(\hatX_t^N)_{t\in[t_1,t_2)}$ by setting $\hatX_{t_1}^N := \xi_{t_1-}^N$ and 
	\begin{equation*}
	\hatX_t^N := X_0 + \frac{\hatX_{t_1}^N - \hatX_{t_1-}^N}{\Gamma(t_1, t_1)}\,\Gamma(t,t_1), \qquad t\in[t_1, t_2).
	\end{equation*}
	Then, we define $(\xi_t^N,B^N)_{t\in[t_1, t_2)}$ as a continuous solution of 
	\begin{equation*}
	\xi_t^N = \hatX_{t_2-}^N + \int_{t_1}^t {\Gamma(t_2, t_2)\,\bigl(b(\xi_{s}^N)\,\ud s + \sigma(\xi_{s}^N)\,\ud B^N_s\bigr)}, 
	\end{equation*}
	for $t\in[t_1, t_2)$ where, by continuity of $\Gamma$,  $\hatX_{t_2-}^N = X_0 + \frac{\xi_{t_1-}^N - X_0}{\Gamma(t_1, t_1)}\Gamma(t_2, t_1)$ is $\cF^N_{t_1}$-measurable. Strictly speaking, this requires to consider an extension of the probability space $(\Omega^N,\cF^N,(\cF^N_t)_{t\in [ 0,t_1]},\P^N)$ (see e.g.~\cite[Definition II.7.1]{IW89}) in order to support the random process $(\xi^N_t,B^N_t)_{t\in[t_1,t_{2})}$.  By an abuse of notation, we still denote  by $(\Omega^N,\cF^N,(\cF^N_t)_{t\in[0,t_{2}]},\P^N)$  the extended probability space.

	\item[\underline{$k\geq2$:}] We now assume that we have constructed by iteration, for $k<N$, a probability space $(\Omega^N,\cF^N,(\cF^N_t)_{t\in [ 0,t_k]},\P^N)$ with a Brownian motion $(B_t^N)_{t\in[t_0, t_k)}$ and processes $(\hatX_t^N)_{t\in[t_0, t_k)}$ and $(\xi_t^N)_{t\in[t_0, t_k)}$. As for the case $k=1$, we set $\hatX_{t_k}^N := \xi_{t_k-}^N$ and define $(\hatX_t^N)_{t\in[t_k, t_{k+1})}$ as
	\begin{equation}\label{eq:hatX}
	\hatX_t^N := X_0 + \sum_{j=1}^k \frac{\hatX_{t_j}^N - \hatX_{t_j-}^N}{\Gamma(t_j, t_j)}\,\Gamma(t,t_j), \qquad t\in[t_k, t_{k+1}).
	\end{equation}
	We observe that $\hatX_{t_{k+1}-}^N = X_0 + \sum_{1 \leq j \leq k} \frac{\hatX_{t_j}^N - \hatX_{t_j-}^N}{\Gamma(t_j, t_j)}\Gamma(t_{k+1},t_j)$ is $\cF^N_{t_k}$-measurable,
	and we define $(\xi_t^N,B^N)_{t\in[t_k, t_{k+1})}$ as a continuous solution of 
	\begin{equation}\label{eq:SDE}
	\xi_t^N = \hatX_{t_{k+1}-}^N + \int_{t_k}^t {\Gamma(t_{k+1}, t_{k+1})\,\bigl(b(\xi_{s}^N)\,\ud s + \sigma(\xi_{s}^N)\,\ud B^N_s\bigr)}, \ t\in[t_k, t_{k+1}). 
	\end{equation}
	If $k=N-1$, we finally define $\hatX_{t_N}=\xi_{t_N}^N=\xi_{t_N-}^N$.
\end{enumerate}
To sum up, we have thus constructed a filtered probability space $(\Omega^N,\cF^N,(\cF^N_t)_{t\in [ 0,T]},\P^N)$ with a Brownian motion~$B^N$ and processes $\hatX^N$, $\xi^N$ that satisfy~\eqref{eq:hatX} and~\eqref{eq:SDE} for any $k\in \{0,\dots,N-1\}$. 

As $\hatX^N$ is expected to converge to a solution $X$ of the SVE \eqref{eq:SVE} as $N \to \infty$, to obtain a $\mathscr{C}$-valued solution $X$, it suffices to prove that $\hatX^N$ remains in $\mathscr{C}$ for all $N \in \mathbb{N}$. Provided that $\hatX^N_{t_{k+1}-} \in \mathscr{C}$,  this reduces to showing that the SDE \eqref{eq:SDE} admits a $\mathscr{C}$-valued solution $\xi^N$ on each interval $[t_k, t_{k+1})$, which is a standard invariance/viability problem for SDEs, see for instance \cite*{abi2019stochastic,bardi2002geometric,da2004invariance,da2007stochastic, doss1977liens}, see Appendix~\ref{A:invarianceSDE}. Interestingly, the coefficients $(b,\sigma)$ and the kernel $\Gamma$ exhibit a distinct decoupling in the SDE \eqref{eq:SDE}, which can be exploited as follows to determine the good assumptions for invariance:

\textbf{Conditions on $(b,\sigma)$:}  
Given that $\hatX^N_{t_{k+1}-} \in \mathscr{C}$, the invariance of $\xi^N$ is ensured by establishing  stochastic invariance of $\mathscr{C}$ for the auxiliary SDE:
\begin{equation}\label{SDE_xi_lambda}
	\xi_t^{\lambda,x} = x + \int_0^t \lambda\,b(\xi_s^{\lambda,x})\,\ud s + \int_0^t \lambda\,\sigma(\xi_s^{\lambda,x})\,\ud B_s, 
\end{equation}
for all
\begin{align}
    \lambda \in \{\Gamma(t,t): t \in [0,T]\}.
\end{align}
We will impose the following assumption:
\begin{equation}\tag{${\bf SDE}_\lambda(\mathscr{C})$}\label{WSI_lambda}
\text{For any } x \in \mathscr{C}, \text{ there exists a weak solution } \xi^{\lambda,x} \text{ to~\eqref{SDE_xi_lambda} such that } \P(\xi_t^{\lambda,x} \in \mathscr{C}, \forall t \geq 0) = 1.
\end{equation}
Equivalent conditions to \eqref{WSI_lambda} in terms of $(b,\sigma)$ are given in Appendix~\ref{A:invarianceSDE}.

\textbf{Conditions on $\Gamma$:}  
to ensure $\hatX^N_{t_{k+1}-} \in \mathscr{C}$ via a recursion on the definition of $\hatX^N$ in \eqref{eq:hatX}, since this only involves prior values of $\hatX^N$ and the kernel $\Gamma$. {For an initial condition $X_0=0$ and $\mathscr{C}=\R^d_+$,}  this naturally leads   to the following class of kernels preserving nonnegativity.

\begin{definition}\label{def:positivity} Let $T>0$. A function  $\Gamma:\Delta_T \to \R_+$ 
(called double kernel) is said to preserve nonnegativity on $[0,T]$ if, 
	for any $K\in \N^*$ and any $x_1,\dots,x_K \in \R$ and $0\le t_1< \dots<t_K<T$ such that
	\begin{align}\label{eq:defnonneg}
	\forall k \in \{1, \dots, K\},\ \sum_{k'=1}^k x_{k'}\Gamma(t_k,t_{k'})\ge 0,  
	\end{align}
	we have $\forall t \in[0,T], \sum_{k :  t_k \le t} x_k \Gamma(t,t_k) \ge 0$. 
	A double kernel $\Gamma:\Delta \to \R_+$ is said to preserve nonnegativity if it satisfies this property for all $T>0$.
\end{definition}
\noindent We will mostly deal with kernels that preserve nonnegativity. However, we will use at some point a time-reversal  {(see the Riccati equation~\eqref{RicVolSqrt})} for which we need to use this notion on $[0,T]$ instead of~$\R_+$. 

\begin{remark}\label{rk_whyR+}
    In general, kernels involved in Stochastic Volterra Equations may be $\R$-valued. Suppose that $\Gamma: \Delta_T \to \R$ satisfies the property of Definition~\ref{def:positivity}. Then taking $K=1$, we see that for any $s\le t\le T$, $\Gamma(t,s)$ and $\Gamma(s,s)$ have necessarily the same sign. Besides, $\Gamma(s,s)=0 \implies \Gamma(t,s)=0$ for $t\ge s$. Let $A=\{s \in [0,T]: \Gamma(s,s)>0\}$, we thus have $\Gamma(t,s)=(2\mathbf{1}_{A}(s)-1)|\Gamma|(t,s)$, and we easily see also that $|\Gamma|$ preserves nonnegativity. From a mathematical point of view, signed kernels that preserve nonnnegativity are thus trivially obtained from the nonnegative ones. For practical applications, one typically expects $\Gamma(s,s)$ to have a constant sign. For these reasons, in Definition ~\ref{def:positivity} we directly assume  that $\Gamma$ takes nonnegative values and is positive on the diagonal, as in the next assumption.  
\end{remark}

{To deal with the case $X_0 \neq 0$ and more general convex domains $\mathscr{C}$, we impose an additional monotonicity condition on $\Gamma$, see Proposition~\ref{prop_pos_gen}.}   For the scheme, we will require the following condition on $\Gamma$. 
\begin{assumption}\label{ass:preserving_nonnegativity}
The kernel $\Gamma : \Delta \to \mathbb{R}_+$ is continuous, satisfies $0<\Gamma(s, s) < \infty$ for all $s \geq 0$, preserves nonnegativity, and for all $s \geq 0$, the map $[s, +\infty) \ni t \mapsto \Gamma(t, s)$ is nonincreasing.
\end{assumption}

\noindent Under Assumptions~\ref{ass:preserving_nonnegativity} and~\eqref{WSI_lambda} for $\lambda \in \{\Gamma(t,t): t \in [0,T]\}$, the processes $\hatX^N$ and $\xi^N$ are well defined and remain in $\mathscr{C}$ as shown in the following key lemma.

\begin{lemma}\label{lem:nonnegative_approx}
	Let $\mathscr{C}$ be a nonempty closed convex domain and $X_0 \in \mathscr{C}$. Let Assumptions~\ref{ass:preserving_nonnegativity} and~\eqref{WSI_lambda} hold true for any $\lambda \in \{ \Gamma(t,t), t \in [0,T] \}$. Then, there exist c\`adl\`ag processes $\hatX^N = (\hatX_t^N)_{t\in[0,T]}$ and $\xi^N=(\xi_t^N)_{t\in[0,T]}$ satisfying \eqref{eq:hatX} and \eqref{eq:SDE} such that
$$\P(\hatX_t^N\in\mathscr{C},\forall t\in[0,T])=1 \quad \text{ and } \quad \P(\xi_t^N\in\mathscr{C},\forall t\in[0,T))=1.$$
\end{lemma}
\begin{proof}
	We have to show that we can construct the processes $\hatX^N$ and $\xi^N$, so that they stay in~$\mathscr{C}$.	We proceed as in the proof of \cite[Theorem 3.5]{Alfonsi23} and we show by induction on $k\in\{1,\ldots,N\}$ that $\P(\hatX_t^N\in\mathscr{C},\forall t\in[0,t_k])=1$ and $\P(\xi_t^N\in\mathscr{C},\forall t\in[0,t_k))=1$ as follows.

	For $k=1$, since $X_0\in \mathscr{C}$, we trivially have $\P(\hatX_t^N\in \mathscr{C},\forall t\in[0, t_1))=1$ by construction. From~\eqref{WSI_lambda} with $\lambda=\Gamma(t_1,t_1)$, we get that there exists a weak solution $\xi$ such that  $\P(\xi_t^N \in \mathscr{C}, t\in[0, t_1))=1$. We then have $\hatX_{t_1}^N := \xi_{t_1-}^N\in\mathscr{C}$ almost surely.
	
	Suppose now that $\P(\hatX_t^N\in \mathscr{C},\forall t\in[0,t_k])=1$ for $k\geq1$. By using Equation~\eqref{eq:hatX}, we write
	\begin{equation*}
	\hatX_{t}^N = X_0 + \sum_{j=1}^{k} \frac{\hatX_{t_j}^N - \hatX_{t_j-}^N}{\Gamma(t_j,t_j)}\,\Gamma(t,t_j), \qquad \text{for all } t\in[t_k,t_{k+1}),
	\end{equation*} 
	Since $\mathscr{C}$ is a nonempty closed convex subset, we can write it as a countable intersection of half-spaces:
	\begin{equation*}
		\mathscr{C}=\bigcap_{\theta \in \Theta} \bigl\{x \in \R^d, \  \alpha_\theta \cdot x +\beta_\theta \ge 0 \bigr\},
	\end{equation*}
	where $\Theta$ is a countable index set, $\alpha_\theta \in \R^d$, $\beta_\theta \in \R$ and $\cdot$ is the scalar product. We obtain
	$$ \alpha_\theta \cdot \hatX_{t}^N +\beta_\theta= \alpha_\theta \cdot X_0 +\beta_\theta + \sum_{j=1}^{k} \frac{\alpha_\theta \cdot(\hatX_{t_j}^N - \hatX_{t_j-}^N)}{\Gamma(t_j,t_j)}\,\Gamma(t,t_j). $$
	From the induction hypothesis, we have $\hatX_{t_l}^N\in \mathscr{C}$ for $l\le k$ and thus $\alpha_\theta \cdot \hatX_{t_l}^N +\beta_\theta= \alpha_\theta \cdot X_0 +\beta_\theta + \sum_{j=1}^{l} \frac{\alpha_\theta \cdot(\hatX_{t_j}^N - \hatX_{t_j-}^N)}{\Gamma(t_j,t_j)}\,\Gamma(t_l,t_j)\ge 0$. We make then use of Proposition \ref{prop_pos_gen}, using Assumption~\ref{ass:preserving_nonnegativity}, to get $\P(\alpha_\theta \cdot \hatX_{t}^N +\beta_\theta  \ge 0,\forall t\in[t_k, t_{k+1}))=1$ and thus $\P( \hatX_t^N\in\mathscr{C},\forall t\in[t_k, t_{k+1}))=1$.  We get in particular that $\hatX_{t_{k+1}-}^N \in \mathscr{C}$ a.s. and by using~\eqref{WSI_lambda} with $\lambda=\Gamma(t_{k+1},t_{k+1})$, there exists a weak continuous solution $\xi^N = (\xi_t^N)_{t\in[t_k, t_{k+1})}$ of Equation~\eqref{eq:SDE} that satisfies $\P(\xi_t^N \in \mathscr{C},t\in[t_k, t_{k+1}))=1$. This  yields $\hatX_{t_{k+1}}^N := \xi_{t_{k+1}-}^N \in \mathscr{C}$ almost surely, and concludes the proof of the induction step.
\end{proof}

  {
\begin{remark}\label{rk:matrix_kernels}
    Let us consider, for this remark only, a matrix valued kernel $\Gamma : \Delta_T \to \mathcal{M}_d(\R)$ that is continuous and such that $\Gamma(t,t)$ is invertible for all $t \in [0,T]$. Then, we can define in an analogous way the following scheme with~\eqref{eq:SDE} and
    \begin{equation*}
	\hatX_{t}^N = X_0 + \sum_{j=1}^{k} \Gamma(t,t_j) \Gamma(t_j,t_j)^{-1} \left(\hatX_{t_j}^N - \hatX_{t_j-}^N \right), \qquad \text{for all } t\in[t_k,t_{k+1}),
	\end{equation*} 
    instead of~\eqref{eq:hatX}.  Let us suppose that $\mathscr{D}\subset \R^d$ is a domain such that:
    \begin{itemize}
        \item For any matrix $M \in \{\Gamma(t,t), t\in[0,T] \}$ and any $x\in \mathscr{D}$, the SDE $\xi^{M,x}_t=x+\int_0^tM b(\xi^{M,x}_s) \ud s + \int_0^tM \sigma(\xi^{M,x}_s) \ud B_s  $ has a weak solution that stays in~$\mathscr{D}$,
        \item For any $x_0 \in \mathscr{D}$,  $K\in \N^*$ and any $x_1,\dots,x_K \in \R^d$ and $0\le t_1< \dots<t_K<T$ such that
	\begin{align*}
	\forall k \in \{1, \dots, K\},\ x_0+ \sum_{k'=1}^k \Gamma(t_k,t_{k'})x_{k'} \in \mathscr{D},  
	\end{align*}
	we have $\forall t \in[0,T], x_0+ \sum_{k :  t_k \le t}  \Gamma(t,t_k)x_k \in \mathscr{D}$. 
    \end{itemize}
    Then, following the proof of Lemma~\ref{lem:nonnegative_approx}, we get $ \P(\hatX_t^N\in\mathscr{C},\forall t\in[0,T])=1$  and $\P(\xi_t^N\in\mathscr{C},\forall t\in[0,T))=1$. Passing to the limit as $N\to \infty$, we may then get the weak existence of $\mathscr{D}$-valued SVEs. However, the difficulty is of course to exhibit relevant domains $\mathscr{D}$ and matrix kernels $\Gamma$ that satisfy the second condition.  This is left for further research.  Note that when $\Gamma$ is scalar and satisfies Assumption~\ref{ass:preserving_nonnegativity}, this condition is satisfied for any closed convex set~$\mathscr{D}$, which is the framework of this paper. 
\end{remark}
}

\begin{remark}\label{rk_approx_scheme_Gamma}
    The above scheme is well adapted to work with the nonnnegativity preserving assumption, as illustrated in the proof of Lemma~\ref{lem:nonnegative_approx}. However, it requires to have $\Gamma(t,t)\not= 0$.  This is not a practical issue for nonnegativity preserving kernels in view of Remark~\ref{rk_whyR+} but may be a limitation for general kernels. It is however possible to define in a similar manner the following scheme, for $k\ge 0$, $ t\in[t_k, t_{k+1})$,
    \begin{align}
        \label{eq:checkX}
	\check{X}_t^N &= X_0 + \sum_{j=1}^k \Gamma(t,t_j)\int_{t_{j-1}}^{t_j} {\bigl(b(\check{\xi}_{s}^N)\,\ud s + \sigma(\check{\xi}_{s}^N)\,\ud \check{B}^N_s\bigr)},\\
    \check{\xi}_t^N &= \check{X}_{t_{k+1}-}^N + \int_{t_k}^t {\bigl(b(\check{\xi}_{s}^N)\,\ud s + \sigma(\check{\xi}_{s}^N)\,\ud \check{B}^N_s\bigr)}, \ t\in[t_k, t_{k+1}).
    \end{align}
 We will use this approximation scheme in the proof of Theorem~\ref{thm_weak_Rd}.
\end{remark}

\subsection{Main existence result of $\mathscr{C}$-valued solution.}

Now that we proved that the approximation scheme stays in~$\mathscr{C}$ in Lemma~\ref{lem:nonnegative_approx}, it suffices to establish convergence, while also relaxing the assumptions on the kernel to allow for singularities on the diagonal. We consider the following assumption.

\begin{assumption}\label{ass:conditions_kernel_limweak} 
The kernel $\Gamma:\Delta \to \R_+$ satisfies Assumption~\ref{ass:conditions1_kernel},
 and there exists a sequence of kernels $(\Gamma_M)_{M \in \N}$ such that each $\Gamma_M$ satisfies Assumption~\ref{ass:preserving_nonnegativity}, 
\begin{align}\label{eq:L2convergence}
    \int_0^t(\Gamma(t,s)-\Gamma_M(t,s))^2 \ud s \to_{M\to \infty} 0, \quad t \geq 0.
\end{align} and for every $T\in(0,+\infty)$, there exist $\eta>0$, $\gamma\in(0,1/2]$ such that for all $M\in\N$,
\begin{equation}\label{eq:GammaMestimate}
	\int_s^t {\Gamma_M(t,u)^2\,\ud u} + \int_0^s {(\Gamma_M(t,u) - \Gamma_M(s,u))^2\,\ud u} \leq \eta\,|t-s|^{2\gamma}, \quad (t,s)\in\Delta_T.
	\end{equation}
\end{assumption}  
Let us observe that for a nonnegative kernel that satisfies Assumption~\ref{ass:conditions1_kernel}, we already know by Lemma~\ref{lem_approx_kernel} that we can find a sequence of continuous  kernels $\Gamma_M$ that are positive on the diagonal and such that~\eqref{eq:L2convergence}
and~\eqref{eq:GammaMestimate} hold. Assumption~\ref{ass:conditions_kernel_limweak} requires in addition that this approximating family is made with continuous, nonnegativity preserving kernels,  which are non-increasing with respect to their first variable.

 {\begin{remark}
    Let us note that if a kernel~$\Gamma$ satisfy Assumptions~\ref{ass:conditions1_kernel} and~\ref{ass:preserving_nonnegativity}, then it satisfies Assumption~\ref{ass:conditions_kernel_limweak} (one can take simply $\Gamma_M=\Gamma$). However, assuming Assumption~\ref{ass:preserving_nonnegativity} may be too strong since it excludes important kernels such as fractional kernels.    We show later in Proposition~\ref{prop_approx_nonnneg} that Assumption~\ref{ass:conditions_kernel_limweak} is satisfied by the family of completely monotone double kernels introduced in Definition~\ref{def:completely_monotone_double}.\\
    In this paper, Assumption~\ref{ass:conditions1_kernel} is required for the weak existence of SVEs (Theorem~\ref{thm_weak_Rd}) while Assumption~\ref{ass:conditions_kernel_limweak} is in force to get weak existence of SVEs in a convex domain (Theorem~\ref{thm_maininvariance}). 
\end{remark}}

We arrive to our main theorem  of  existence  of weak $\mathscr{C}$-valued solutions to the SVE \eqref{eq:SVE}. Now that we are allowing singularities of the kernel on the diagonal, we introduce the following set: 
\begin{equation}
\label{eq:DT_definition}
\Lambda_T = 
\begin{cases} 
\{\Gamma(t, t) : t \in [0, T]\} & \text{if } \Gamma \text{ is continuous on } \Delta_T,\\\mathbb{R}_+ &  
 \text{otherwise.}
\end{cases}
\end{equation}

\begin{theorem}\label{thm_maininvariance}
	Let $T>0$. Let $\mathscr{C}$ be a nonempty closed convex domain and $X_0 \in \mathscr{C}$. Let $\Gamma$ satisfy Assumption~\ref{ass:conditions_kernel_limweak}, and $b,\sigma$ satisfy Assumption~\ref{ass:conditions1_coefficients}. We assume that~\eqref{WSI_lambda} holds for any $\lambda \in \Lambda_T$. Then, there exists a  weak continuous solution to the SVE \eqref{eq:SVE} that stays in~$\mathscr{C}$ on $[0,T]$.  {In addition, for any $T>0$, the paths of $X$ on $[0,T]$ are Hölder continuous of any order less than $\gamma$, where $\gamma$ is the constant associated with $\Gamma$ and $T$ in Assumption \ref{ass:conditions1_kernel}. }
\end{theorem}

\begin{proof}
    The proof is given in Section~\ref{S:proofs_main_existence}.
\end{proof}

\begin{remark}
	In Assumption~\ref{ass:conditions_kernel_limweak}, we suppose that the approximating family $\Gamma_M$ preserves nonnegativity on~$\R_+$. In a straightforward manner, it is in fact sufficient to assume that these kernels  preserve nonnegativity on $[0,T]$ to get $\P(X_t \in \mathscr{C}, t\in [0,T])=1$ in Theorem~\ref{thm_maininvariance}.
\end{remark}
  {\begin{remark}
    Let $b:\mathscr{C}\to \R^d$ and $\sigma:\mathscr{C}\to \R^d$ be coefficients that satisfy Assumption~\ref{ass:conditions1_coefficients}, but for $x\in \mathscr{C}$. Let $\Pi_\mathscr{C}$ denote the $L^2$-projection on the closed convex set $\mathscr{C}$. Then, $b\circ \Pi_\mathscr{C}$ and $\sigma\circ \Pi_\mathscr{C}$ are defined on $\R^d$ and satisfy Assumption~\ref{ass:conditions1_coefficients}.
\end{remark}}
\begin{remark}\label{Rk_time_inhomogeneous} To lighten notation, we have preferred to present our results for SVEs with time homogeneous coefficients,   { with a Brownian motion of the same dimension as $X$}.  However, the results of Theorems~\ref{thm_weak_Rd} and~\ref{thm_maininvariance} can be easily extended to the case of time inhomogeneous coefficients $b:\R_+ \times \R^d\to \R^d$ and    {$\sigma:\R_+ \times \R^d\to \mathcal{M}_{d'}(\R)$ with $d' \ge 1$} such that
$$\forall t\ge 0,\forall x \in \R^d, |b(t,x)|+|\sigma(t,x)|\le C_{LG}(1+|x|).$$
  {Let $B'$ denote a $d'$-dimensional Brownian motion.} If $\Gamma$ satisfies Assumption~\ref{ass:conditions1_kernel}, there exists a weak solution to 
\begin{equation}\label{SVE_timedep}
    X_t= X_0 + \int_0^t  \Gamma(t,s)b(s,X_s) \ud s +\int_0^t  \Gamma(t,s) \sigma(s,X_s) \ud B'_s,
\end{equation} 
that is H\"older continuous on every $[0,T]$, $T>0$. Besides, if  $\Gamma$ satisfies Assumption~\ref{ass:conditions_kernel_limweak}  and $\mathscr{C}\subset \R^d$ is a nonempty closed convex set such that for $T>0$ and any $t \in [0,T]$, $x\in \mathscr{C}$ and $\lambda\in \Lambda_T$, there exists a weak solution to the SDE
$$\zeta^{\lambda, x,t}_u=x +\int_t^u \lambda b(s,\zeta^{\lambda, x,t}_s) \ud s +\int_t^u \lambda \sigma (s,\zeta^{\lambda, x,t}_s) \ud B'_s, \ u \in[t,T],$$
such that $\P(\zeta^{\lambda, x,t}_u \in \mathscr{C}, u \in [t,T])=1$, then there exists an H\"older continuous weak solution to~\eqref{SVE_timedep} on  $[0,T]$ such that $\P(X_t \in \mathscr{C}, t \in [0,T])=1$. 

Note that the extension to time inhomogeneous coefficients allows to deal with input curves $g_0:\R_+\to \R^d$. For example when $g_0(t)=X_0+\int_0^t \Gamma(t,s) h_0(s) \ud s$, $X$ is a weak solution of $X_t= g_0(t) + \int_0^t  \Gamma(t,s)\tilde{b}(s,X_s) \ud s +\int_0^t  \Gamma(t,s) \sigma(s,X_s) \ud B'_s$ if and only if it is a weak solution of
$X_t= X_0 + \int_0^t  \Gamma(t,s)b(s,X_s) \ud s +\int_0^t  \Gamma(t,s) \sigma(s,X_s) \ud B'_s$ with $b(t,x)=\tilde{b}(t,x)+h(t)$, and there exists a H\"older continuous solution in~$\mathscr{C}$ on $[0,T]$ under the above hypotheses.  

\end{remark}

 {\begin{remark}
    We note that Theorem~\ref{thm_maininvariance} reduces the invariance problem for $\mathcal{C}$ in the Volterra equation \eqref{eq:SVE} to an invariance problem for an infinite family of stochastic differential equations (see \ref{WSI_lambda}), under suitable structure on the kernel.  A related splitting strategy has appeared using a different approach in \cite{CT19}, for  the case $\mathcal{C}=\mathbb{R}_+$, where the Volterra dynamics are decomposed into a classical infinite-dimensional SDE together with a variation-of-constants type integral term which would require suitable structure on the kernel to preserve non-negativity, see \cite[Theorem 4.17]{CT19}.  The approach there relies on the so-called multifactor Markovian lift. We refer to Section~\ref{S:lift} for more details on such lifts.
\end{remark}}

The next corollary strengthens the weak existence result of Theorem~\ref{thm_maininvariance} to a strong existence and uniqueness under additional Lipschitz conditions on the coefficients $(b,\sigma)$ using the powerful and generic framework of \cite{K:14}.  We say that a \textit{weak continuous solution} $X$ is a \emph{strong continuous solution} if it is adapted to the augmented natural filtration generated by $B$.

\begin{corollary}\label{cor_strong_domain}
    Let $T>0$. Let $\mathscr{C}$ be a nonempty closed convex domain and $X_0 \in \mathscr{C}$. Let $\Gamma$ satisfy Assumption~\ref{ass:conditions_kernel_limweak}, and $b,\sigma$ be Lipschitz continuous, i.e.~there exists $C>0$ such that 
    $$ |b(x) - b(y)|  + |\sigma(x) - \sigma(y)|\leq C |x - y|, \quad x,y\in \R^d.$$  We assume that~\eqref{WSI_lambda} holds for any $\lambda \in \Lambda_T$. Then, there exists a unique strong continuous solution to the SVE \eqref{eq:SVE} that stays in~$\mathscr{C}$ on $[0,T]$.  {In addition, for any $T>0$, the paths of $X$ on $[0,T]$ are Hölder continuous of any order less than $\gamma$, where $\gamma$ is the constant associated with $\Gamma$ and $T$ in Assumption \ref{ass:conditions1_kernel}. }
\end{corollary}

\begin{proof}
  Using the Lipschitz continuity of the coefficients $(b,\sigma)$ it is straightforward to obtain, through standard estimates,  pathwise uniqueness of solutions to the SVE \eqref{eq:SVE}. An application of Theorem~\ref{thm_maininvariance} yields weak existence. Invoking 
  \citet[Theorem~1.5 and Lemma~2.10]{K:14}, pathwise uniqueness and weak existence imply strong existence and uniqueness and ends the proof. 
\end{proof}

We then conclude this section by giving an application of Theorem~\ref{thm_maininvariance} to the family of non-convolution kernels presented in Exemple \ref{ex:fractional_double}. This leads to the following theorem.
\begin{corollary}\label{cor:fractional_invariance}
	Let $T>0$. Let $\mathscr{C}$ be a nonempty closed convex domain and $X_0 \in \mathscr{C}$. Let $b,\sigma$ satisfy Assumption~\ref{ass:conditions1_coefficients} and \eqref{WSI_lambda} hold for any $\lambda>0$. Further, let $G:\R_+^*\to\R$ be a completely monotone convolution kernel satisfying \eqref{eq:conditions_convolution}, and $h:\R_+^*\to\R_+^*$ be a locally integrable function, locally bounded away from zero and such that $t \mapsto \int_0^t {h(u)\,\ud u}$ is $\beta$-H\"older continuous on $[0,T]$ with $\beta \in (0,1]$. Consider the SVE
	\begin{equation}\label{eq:SVE_fractional}
		X_t = X_0 + \int_0^t {G\biggl(\int_s^t {h(u)\,\ud u}\biggr)\Bigl(b(X_{s})\,\ud s + \sigma(X_{s})\,\ud B_s\Bigr)},
	\end{equation}
	Then, there exists a weak continuous solution to the SVE \eqref{eq:SVE_fractional} that stays in $\mathscr{C}$ on $[0,T]$. In addition, the paths of $X$ on $[0,T]$ are Hölder continuous of any order less than $\beta\,\gamma$, where $\gamma$ is the constant appearing in \eqref{eq:conditions_convolution}. \\
    If, moreover, $b$ and $\sigma$ are Lipschitz continuous, then, there exists a unique strong continuous solution to the SVE \eqref{eq:SVE_fractional} whose paths are Hölder continuous of any order less than $\beta\,\gamma$.
\end{corollary}
\begin{proof}
    We first recall that $\Gamma(t,s) = G(\int_s^t {h(u)\,\ud u})$ satisfies Assumption~\ref{ass:conditions1_kernel} as shown in Example~\ref{ex:fractional_double}. As $G$ is a completely monotone function, by Bernstein's theorem, it can be written as $t\mapsto \int_0^{+\infty} {e^{-x\,t}\,\mu(\ud x)}$ for some Borel measure $\mu$ on $\R_+^*$ finite on compact sets. Then, injecting $\int_s^t {h(u)\,\ud u}$ into it,
	\begin{equation*}
	\Gamma(t,s) = \int_0^{+\infty} {e^{-x\int_s^t {h(u)\,\ud u}}\,\mu(\ud x)},
	\end{equation*}
    and use Proposition~\ref{prop_approx_nonnneg} thereafter to get that $\Gamma$ satisfies Assumption~\ref{ass:conditions_kernel_limweak}. We can finally apply Theorem~\ref{thm_maininvariance} and Corollary~\ref{cor_strong_domain}.
\end{proof}

\section{Explicit specifications of the kernel and coefficients}\label{S:ecamples}
\subsection{Constructing double Volterra kernels that preserve nonnegativity}

\subsubsection{Examples}
We first give a simple example of nonnegativity preserving double kernel that is not of convolution type and extends~\cite[Example 2.2]{Alfonsi23}, we have: 
	\begin{example}\label{ex:exp} Let $\Gamma(t,s) = b(s)c(t) e^{-\rho((s,t])}$ for $(s,t) \in \Delta$ where $\rho$ is a Borel measure on $\R_+$ finite on compact sets, $b:\R_+\to \R_+$ and $c:\R_+\to \R_+^*$.  Then, $\Gamma$ preserves nonnegativity. Indeed, let $x_1,\ldots, x_K \in \R$  and and $0\le t_1< \dots<t_K$ be such that \eqref{eq:defnonneg} is satisfied. Then, for an arbitrary  $t\geq 0 $ and $k$ such that $t_k\leq t < t_{k+1}$ (convention $t_0=0$), we get 
	\begin{equation*}
		\sum_{k':t_{k'}\leq t} x_{k'}\,\Gamma(t,t_{k'}) =  \sum_{k'=1}^k   x_{k'}\,b(t_{k'})c(t) e^{-\rho((t_{k'}, t])} = \frac{c(t)}{c(t_k)}e^{-\rho((t_{k}, t])}\sum_{k'=1}^k  x_{k'}\,\Gamma(t_k,t_{k'}) \geq 0. 
	\end{equation*}
	Let us note that this example includes the particular case $\rho( \ud u)=f(u)\,\ud u$ where $f:\R_+^*\to\R$ is locally integrable. Then, $\Gamma(t,s) = b(s)c(t)e^{-\int_s^t {f(u)\,\ud u}}$ preserves nonnegativity.
\end{example}

We now present a general way to obtain nonnegativity preserving double kernels from convolution kernels. Section~\ref{Sec_kernels} presents a study  of nonnegativity preserving double kernels, and Theorem~\ref{thm_char_pos} gives a characterization that enables us to obtain the following corollary. 
\begin{corollary}\label{cor_convolution}
	Let $G:\R_+\to \R_+$ be a convolution kernel such that $G(0)>0$ that preserves nonnegativity in the sense of~\cite[Definition 2.1]{Alfonsi23}, i.e. for any $K\in \N^*$,  
	$x_1,\dots,x_K \in \R$ and $0\le t_1< \dots<t_K$ such that
	$$\forall k \in \{1, \dots, K\},\ \sum_{k'=1}^k x_{k'}G(t_k-t_{k'})\ge 0 \implies \forall t \ge 0, \sum_{k :  t_k \le t} x_k G(t-t_k) \ge 0.$$ 
	Let $\rho$ be a Borel measure on $\R_+$ finite on compact sets. Then, the double kernels $\Gamma^r(t,s)=G(\rho((s,t]))$ and $\Gamma^\ell(t,s)=G(\rho([s,t)))$ preserve nonnegativity. 
\end{corollary}
The convolution kernels $G(t)=\sum_{i=1}^n\lambda_i e^{-\rho_i t}$ with $n\ge 1$, $\omega_i> 0$ and $0\le \lambda_1< \dots <\lambda_n$ or $G(t)=c(t + \varepsilon )^{H-1/2}$ with $c,\varepsilon>0$ and $H\in(0,1/2)$ satisfy the assumption of Corollary~\ref{cor_convolution} by~\cite[Theorem 2.1]{Alfonsi23}. Therefore, the double kernels $\sum_{i=1}^n \omega_i\,e^{-\lambda_i\,\rho((s,t])}$ or $c(\rho((s,t]) + \varepsilon)^{H-1/2}$ preserve nonnegativity. If, moreover, $\rho( \ud u)=f(u)\,\ud u$ where $f:\R_+^*\to\R$ is locally integrable, they also satisfy Assumption~\ref{ass:preserving_nonnegativity}.

\subsubsection{Completely monotone double kernels}
A function $K:\R_+^*\to \R_+$ is said to be \emph{completely monotone} if $K \in C^{\infty}(\R_+^*,\R_+)$ such that $(-1)^n\,K^{(n)}\geq0$ for every $n\geq0$. By Bernstein's theorem, this is equivalent to the existence of a Borel measure $\theta$ on $\R_+^*$ finite on compact sets such that $K(t) = \int_0^{+\infty} {e^{-\alpha t}\,\theta(\ud \alpha)}$ for all $t\geq0$. Here, we propose the following generalization of completely monotone functions to double kernels.
\begin{definition}\label{def:completely_monotone_double}
	Let $\Gamma:\mathring{\Delta} \to \R_+$. $\Gamma$ is said to be a \emph{completely monotone double kernel} if there exist 
	\begin{enumerate}
		\item[(i)] a Borel measure $\mu$ on $\R$ finite on compact sets;
		\item[(ii)] a family $(\rho(\alpha,\cdot))_{\alpha\in\R}$ of Borel measures on $\R_+$, finite on compact sets and such that for all $\alpha\leq\beta$, $\rho(\beta, \cdot) - \rho(\alpha, \cdot)$ is a non-negative measure,
	\end{enumerate}  
	such that 
	\begin{equation}\label{eq:Gammamu}
		\Gamma(t,s) = \int_{\R} {e^{-\rho(\alpha, (s,t])}\,\mu(\ud\alpha)} \quad \text{ or } \quad \Gamma(t,s) = \int_{\R} {e^{-\rho(\alpha, [s,t))}\,\mu(\ud\alpha)}, \quad (t,s)\in\mathring{\Delta},
	\end{equation}
\end{definition}
\noindent and $\Gamma(t,s)<\infty$ for $t>s$.  When $\mu(\R)<\infty$, a completely monotone double kernel can be extended on $\Delta$ by taking $\Gamma(t,t)=\mu(\R)$.  

\begin{theorem}\label{thm:completely_monotone_double}
	Let $\Gamma:\Delta \to \R_+$ be a completely monotone double kernel in the sense of Definition \ref{def:completely_monotone_double} such that $0 < \mu(\R) < +\infty$. Then, $\Gamma$ preserves nonnegativity.
\end{theorem}
Let us note that we already know from~\cite[Theorem 2.11]{Alfonsi23} that completely monotone convolution kernels preserves nonnegativity. By Bernstein's theorem, every completely monotone function can be written as $t\mapsto \int_0^\infty e^{-\alpha t} \mu(\ud \alpha)$ for some Borel measure $\mu$ on $\R_+$ finite on compact sets. Using Corollary~\ref{cor_convolution}, we get that $(t,s)\mapsto \int_0^\infty e^{-\alpha \rho([s,t))} \mu(\ud \alpha)$ and $(t,s)\mapsto \int_0^\infty e^{-\alpha \rho((s,t])} \mu(\ud \alpha)$ preserve nonnegativity for any Borel measure $\rho$ on $\R_+$ finite on compact sets. Theorem~\ref{thm:completely_monotone_double}
 thus extends this result to a family of Borel measures $\rho(\alpha,\ud x)$ that may not depend linearly on~$\alpha$. Its proof is postponed to Section~\ref{Sec_kernels}. 

\begin{prop}\label{prop_approx_nonnneg}
    Let $\Gamma(t,s)= \int_{\R} {e^{-\rho(\alpha, (s,t])}\,\mu(\ud\alpha)}$ be a completely monotone double kernel as in Definition~\ref{def:completely_monotone_double} with a family of atomless Borel measure $\rho(\alpha,\cdot)$. If $\Gamma$ satisfies Assumption~\ref{ass:conditions1_kernel}, then it also satisfies Assumption~\ref{ass:conditions_kernel_limweak} with the approximating family 
    $$\Gamma_M(t,s)= \int_{[-M,M]} {e^{-\rho(\alpha, (s,t])}\,\mu(\ud\alpha)}, \ M \in \N^*.$$
\end{prop}
\begin{proof}
First, let us note that the atomless property ensures that $(t,s) \mapsto \rho(\alpha, (s,t])=\rho(\alpha, (s,t))$ is continuous for all $\alpha \in \R$, and therefore the dominated convergence theorem gives the continuity of~$\Gamma$ and $\Gamma_M$ on $\mathring{\Delta}$. Combined with Theorem~\ref{thm:completely_monotone_double},  we get that $\Gamma_M$ satisfies Assumption~\ref{ass:preserving_nonnegativity} since $0<\mu([-M,M])<\infty$ for $M$ large enough ($0<\mu([-M,M])$ ensures that $\Gamma_M(t,t)>0$).
We now observe that we have 
$0\le e^{-\rho(\alpha,(s,t))}-e^{-\rho(\alpha,(u,t))}$ for $0<u<s<t$. We thus get
$$0\le \Gamma_M(s,u)-\Gamma_M(t,u) \le \Gamma(s,u)-\Gamma(t,u) \text{ and } 0\le \Gamma_M(s,u)\le \Gamma(s,u).$$
Combining these inequalities with Assumption~\ref{ass:conditions1_kernel}, we get~\eqref{eq:GammaMestimate}. 

By Assumption~\ref{ass:conditions1_kernel}, $\int_0^t \Gamma(t,s)^2 \ud s <\infty$. Since $0\le \Gamma(t,s)-\Gamma_M(t,s)\le \Gamma(t,s)$ and $\Gamma_M(t,s) \to \Gamma(t,s)$ by the monotone convergence theorem, we get~\eqref{eq:L2convergence} by the dominated convergence theorem.
\end{proof}

\subsection{Examples of domains $\mathscr{C}$ and coefficients $(b,\sigma)$}

We collect in Appendix~\ref{A:invarianceSDE} several characterizations of invariance for SDEs, providing explicit conditions on the coefficients $b$ and $\sigma$ to ensure \eqref{WSI_lambda}.  
In particular, we list specific examples of convex domains, describe the behavior of the coefficients $(b, \sigma)$, and highlight the continuity and potential unboundedness of the kernel $\Gamma$ on the diagonal.  

\begin{enumerate}
    \item \textbf{Non-negative orthant  $\mathscr{C} = \mathbb{R}^d_+$:} For $d = 1$, \eqref{WSI_lambda} holds for any $\lambda > 0$ if $b(0) \geq 0$ and $\sigma(0) = 0$. More generally, for $\mathscr{C} = \mathbb{R}_+^d$, \eqref{WSI_lambda} holds for any $\lambda \in \Lambda_T$ if:
\begin{align}\label{eq:bsigmaRd+}
\text{For any } i \in \{1, \dots, d\}, \text{ $x_i = 0$ implies $b_i(x) \geq 0$ and $\sigma_i(x) = 0, \quad  x \in \mathbb R^d_+$,}    
\end{align}
where $\sigma_i(x)$ is the $i$th row of $\sigma(x)$.  This covers and extends \cite[Theorem 3.4]{AJLP19} derived for convolution kernels and continuous coefficients $(b,\sigma)$ with linear growth. 

\item \textbf{Non-negative symmetric matrices $\mathscr{C} = \mathbb{S}^d_+$:}  
    Using the vectorization operator, the problem can be reduced to $\mathbb{R}^{d^2}$.
    In particular, the affine case is treated more in details in Section~\ref{S:squareroot}. This is valid for continuous kernels $\Gamma$ on $\Delta_T$ as well as unbounded  kernels on the diagonal. 
 In the affine case, Theorem~\ref{thm_maininvariance} establishes the weak existence of an $\mathbb{S}^d_+$-valued solution $X$ to the stochastic Volterra Wishart-type equation:
    \begin{align}
        X_t &= X_0 + \int_0^t \Gamma(t,s) \left( \alpha + M X_s + X_s M^\top  \right) \ud s \\
        &\quad + \int_0^t \Gamma(t,s) \left( \sqrt{X_s} \ud B_s Q + Q^\top \ud B_s^\top \sqrt{X_s} \right),
    \end{align}
    where $B$ is a $d\times d$-matrix Brownian motion, $Q, M$ are $d\times d$-matrices, $X_0 \in \mathbb{S}^d_+$, and the $d\times d$-matrix $\alpha$ satisfies  
    \begin{align}\label{eq:condalpha}
        \alpha - \sup_{t\leq T}\Gamma(t,t)(d-1)QQ^\top \in \mathbb{S}^d_+,
    \end{align}
    ensuring that \eqref{WSI_lambda} holds,  see \cite[Theorem 2.4 and Condition (2.4)]{cuchiero2011affine}, but only for continuous kernels $\Gamma$ on $\Delta_T$.  
    If the kernel $\Gamma$ is unbounded on the diagonal, condition \eqref{eq:condalpha} cannot hold.  
    In the context of Volterra equations, this result is new, even for convolution kernels.  

    \item \textbf{Extension to non-convolution cases:}  
    Our results also allow to  extend to the non-convolution case other (polynomial) Volterra processes with possibly unbounded kernels that remain confined to the unit ball or compact intervals, as derived in \cite[Theorem 2.7 and Corollary 2.8]{abi2024polynomial}.  For example,  Theorem~\ref{thm_maininvariance} gives the weak existence of a $[0,1]$-valued solution $X$ to the stochastic Volterra Wright-Fisher-type equation:
    $$X_t=X_0+\int_0^t \Gamma(t,s) (a+bX_s) \ud s +\int_0^t \Gamma(t,s) \sigma \sqrt{X_s(1-X_s)} \ud B_s, $$
    with $\sigma>0$, $X_0\in[0,1]$, $a\ge 0$ and $a+b\le 0$.
\end{enumerate}
 {Of course, this list of examples is not exhaustive and we refer to~\cite{SpVe} and~\cite{CKRMT} for further examples of symmetric cones and polyhedral state space, on which we may define SVEs. }

 {\subsection{Link with multifactor Markovian lifts.}\label{S:lift}
For a double kernel $\Gamma$ completely monotone in the sense of Definition~\ref{def:completely_monotone_double}, of the form
\[
\Gamma(t,s) = \int_{\mathbb{R}} e^{-\rho(\alpha,(s,t])}\,\mu(\mathrm{d}\alpha),
\]
one can lift the stochastic Volterra equation into a (possibly infinite-dimensional) Markovian system. For convolution kernels, this approach has been first implemented for fractional Brownian motion by \cite{carmona1998fractional}  and more recently extended to stochastic Volterra equations by \cite{AJEE19b,CT19,harms2019affine}.\\
Let $X$ be a solution to \eqref{eq:SVE}. Under suitable integrability assumptions on the measure $\mu$, an application of the stochastic Fubini theorem yields the representation
\[
X_t = X_0 + \int_{\mathbb{R}} Y_t^{\alpha}\,\mu(\mathrm{d}\alpha),
\]
where, for each $\alpha \in \mathbb{R}$, the factor process $(Y_t^{\alpha})_{t \geq 0}$ satisfies time-inhomogeneous Markovian dynamics of the form
\begin{equation}\label{eq:yalpha}
\begin{aligned}Y_t^{\alpha} &= \int_0^t e^{-\rho(\alpha,(s,t])} b\left(X_0 + \int_{\mathbb{R}} Y_s^{\alpha}\,\mu(\mathrm{d}\alpha)\right)\,\mathrm{d}s \\
&\quad \quad \quad  \quad + \int_0^t e^{-\rho(\alpha,(s,t])} \sigma\left(X_0 + \int_{\mathbb{R}} Y_s^{\alpha}\,\mu(\mathrm{d}\alpha)\right)\,\mathrm{d}W_s.    
\end{aligned}
\end{equation}
This representation provides a multifactor Markovian lift of the original Volterra dynamics, where the memory is encoded through the continuum of factors indexed by $\alpha$.\\ 
A related question that has attracted growing interest in the literature concerns the invariant domain of the lifted process $(Y_t^{\alpha})_{\alpha \in \mathrm{supp}(\mu)}$. In the case where $\mathcal{C}=\mathbb{R}_+$ and the kernel is of convolution type, \cite{AJEE19b, CT19}  identify invariant domains for $Y$. For discrete measures $\mu$, more concrete descriptions of the resulting state space have recently been obtained by \cite*{abijaber2024state}.  In practice, finite-dimensional approximations  can be obtained using \eqref{eq:yalpha} by discretizing the measure $\mu$, by weighted sums of Dirac measures, leading to tractable multifactor models, as done for instance in \cite{AJEE19a} for convolution kernels.}
\section{Weak existence and Uniqueness of Volterra square-root processes with non-convolution kernels}\label{S:squareroot}
In this section, we establish the weak existence and uniqueness of an affine Volterra square-root process with a non-convolution kernel. This extends the results in \cite*[Section 6]{AJLP19} to the non-convolution case. 

Let $\Gamma$ be  a scalar kernel and a $d$-dimensional vector $b^0$ and a $d\times d$-matrix $B$ such that 
\begin{equation} \label{sqrt2}
b^0 \in \mathbb{R}^d_+ \text{ and } B_{ij} \ge 0 \text{ for } i \ne j.
\end{equation}
We call  {Volterra square-root process} any $\R^d_+$-valued solution of the   equation
\begin{equation} \label{eq:VolSqrt}
\begin{aligned}
X_{i,t} &= X_{i,0} + \int_0^t \Gamma(t,s)\left(b_i^0 + (BX)_{i,s} \right) \ud s  + \int_0^t \Gamma(t,s) \sigma_i \sqrt{X_{i,s}} \ud  W_{i,s}, \quad i = 1, \ldots, d,
\end{aligned}
\end{equation}
where $\sigma_i>0$ and $W$ is a $d$-dimensional Brownian motion. 

The following theorem establishes the weak existence and uniqueness of $\R^d_+$ solutions to~\eqref{eq:VolSqrt}, together with an expression for their Laplace transform
\begin{align}
    \mathbb E \left[ \exp\left( \int_t^T f(s)^\top  X_s \ud  s  \right) \Big|  \mathcal F_t \right], \quad f:[0,T] \to \R^d_-,  
\end{align}
in terms of the following
Riccati--Volterra equation \begin{equation} \label{RicVolSqrt}
\begin{aligned}
\psi(t) &=  \int_t^T \Gamma(s,t) F(s,\psi(s)) \ud  s , \\
F_i(s, \psi) &= \left( f_i(s) + (B^\top \psi)_i + \frac{\sigma_i^2}{2} \psi_i^2 \right),
 \quad i = 1, \ldots, d.
\end{aligned}
\end{equation}

We will need the Volterra kernel $\tilde \Gamma: \Delta_T \to \R_+$  defined by
\begin{align}\label{eq:tildeGamma}
     \tilde \Gamma (t,s) = \Gamma (T-s, T-t), \quad 0\le  s\leq t\le T.  
    \end{align}

\begin{remark}\label{R:riccatifwd}
The Riccati--Volterra \eqref{RicVolSqrt} is written in backward form. It can be re-expressed, using a change of variables,  in the following forward form on $\tilde \psi(t) := \psi({T-t})$:
\begin{align}\label{eq:RiccatiFwd}
    \tilde \psi(t) = \int_0^t \tilde \Gamma(t,s) F(T-s, \tilde \psi(s))  \ud  s, \quad t \leq T,  
\end{align}
    where the Volterra kernel $\tilde \Gamma: \Delta_T \to \R_+$ is defined by \eqref{eq:tildeGamma}.
    In particular, if $\Gamma$ is a convolution kernel of the form $\Gamma(t,s)=1_{s\leq t}K(t-s)$, then, $\tilde \Gamma(t,s)= 1_{s\leq t}K(t-s) =\Gamma(t,s) $ and one recovers from \eqref{eq:RiccatiFwd}, the Riccati-Volterra  \cite[equation (6.3)]{AJLP19}. 
\end{remark}

\begin{theorem} \label{T:VolSqrt}
Fix  $b^0$ and $B$ as in \eqref{sqrt2}, and $T>0$. Assume that $\Gamma$ and $\tilde \Gamma $ defined in \eqref{eq:tildeGamma} satisfy Assumption~\ref{ass:conditions_kernel_limweak}.
\begin{enumerate}
\item\label{T:VolSqrt:1} The stochastic Volterra equation \eqref{eq:VolSqrt} has a unique in law $\mathbb{R}^d_+$-valued continuous weak solution $X$ on $[0,T]$ for any initial condition $X_0 \in \mathbb{R}^d_+$. For each $i$, the paths of $X_i$ are Hölder continuous of any order less than $\gamma$, where $\gamma$ is the constant associated with $\Gamma$ in Assumption \ref{ass:conditions1_kernel}.

\item\label{T:VolSqrt:2} For any $f \in C([0,T], \mathbb{R}^d_-)$ the Riccati--Volterra equation \eqref{RicVolSqrt} has a unique global solution $\psi \in C([0,T], \mathbb{R}^d_-)$, i.e.~ $\psi_i \le 0$, $i = 1, \ldots, d$. Moreover, we have the following exponential-affine transform formula 
\begin{align}\label{eq:laplace}
    \mathbb E \left[ \exp\left( \int_t^T f(s)^\top  X_s  \ud  s  \right) \Big|  \mathcal F_t \right] = \exp\left( \int_t^T F(s, \psi(s))^\top  g_t(s) \ud s \right),
\end{align}
 with 
\begin{align}
    g_t(s) &= g_0(s) + \int_0^t \Gamma(s,r)  \ud  Z_r, \quad  t\leq s, \label{eq:gts}\\
    dZ_{i,t} &= (BX_t)_{i,t}  \ud  t + \sigma_i \sqrt{X_{i,t}}  \ud W_{i,t}, \quad 
    g_0(s) = X_0 + \int_0^s \Gamma (s,r)b^0 \ud  r.  
\end{align}
\end{enumerate}
\end{theorem}

\begin{proof}
    The proof is given in Section~\ref{S:proofsquareroot}. 
\end{proof}

\begin{remark}\label{rem:Gamma_tilde}
Assume that $\Gamma$ and $\tilde \Gamma$ satisfy Assumption~\ref{ass:conditions1_kernel} . 
    If $\Gamma$ is completely monotone in the sense of Definition~\ref{def:completely_monotone_double}, then $\tilde \Gamma$ is   clearly  completely monotone and hence by Proposition~\ref{prop_approx_nonnneg}, both $\Gamma$ and $\tilde \Gamma$ satisfy Assumption~\ref{ass:conditions_kernel_limweak} as required in Theorem~\ref{T:VolSqrt}. 
\end{remark}

In the following example, we provide an application of Theorem \ref{T:VolSqrt} to the generalized fractional kernel, namely setting $G(t) = \frac{t^{\alpha-1}}{\Gamma_e(\alpha)}$ within Example \ref{ex:fractional_double}. In this particular setting, we are able to rewrite the Laplace transform of the associated affine Volterra square-root process by expressing it in terms of the solution of a fractional Riccati equation with time-dependent coefficients. This notably extends the existing expressions in \cite[Example 4.7]{AJLP19} and \cite{EER19}.

\begin{example}\label{ex:fractional_double_bis}
	Let us consider the kernel of Example~\ref{ex:fractional_double} where $G(t) = \frac{t^{\alpha-1}}{\Gamma_e(\alpha)}$ with $\alpha\in(\frac{1}{2},1]$ and $h:\R_+^*\to\R_+^*$ defined therein. 
	Fix $b^0$ and $B$ as in \eqref{sqrt2}, and $T>0$. Then, the SVE
	\begin{equation}\label{eq:VolSqrt_fractional}
	\begin{aligned}
	X_{i,t} &= X_{i,0} + \frac{1}{\Gamma_e(\alpha)}\int_0^t {\biggl(\int_s^t {h(u)\,\ud u}\biggr)^{\alpha-1}\Bigl(\bigl(b_i^0 + (BX)_{i,s} \bigr) \ud s  + \sigma_i \sqrt{X_{i,s}}\,\ud  W_{i,s}\Bigr)}, 
	\end{aligned}
	\end{equation}
	for $i = 1, \ldots, d$, where $\Gamma$ denotes the Euler Gamma function, has a unique in law $\R_+^d$-valued continuous weak solution  on $[0,T]$ for any initial condition $X_0 \in \mathbb{R}^d_+$. Besides, the paths of $X$ are Hölder continuous of any order less than $\beta(\alpha-\frac{1}{2})$, where $\beta$ is the H\"older exponent of $t \mapsto \int_0^t {h(u)\,\ud u}$ on $[0,T]$. \\
	With this particular choice of double kernel, the Riccati-Volterra equation~\eqref{RicVolSqrt} can be expressed by means of a fractional Riccati equation. To do so, we recall the Riemann--Liouville fractional integral (resp. derivative) of $f:\R_+^*\to \R$ of order $r\in(0,1]$ defined by  $I^r f(t) := \frac{1}{\Gamma_e(r)}\int_0^t {(t-s)^{r-1}\,f(s)\, \ud s}$ (resp. $D^rf(t) := \frac{\ud }{\ud t}\,I^{1-r}f(t)$ and $D^1f(t)=f'(t)$) for $t > 0$. Let $\tilde{\psi}$ be defined as in~\eqref{eq:RiccatiFwd} and $\varphi=\tilde{\psi}\circ \xi$, with $\xi:[0,\int_0^T h(u)\,\ud u]\to[0,T]$ being the inverse function of $t\mapsto\int_0^t {h(T-u)\,\ud u}$. Then, we can show by tedious but elementary calculations that $\varphi$ solves  the fractional Riccati equation
		\begin{equation}\label{eq:Riccati_fractional}
			D^{\alpha}\varphi = \tilde{F}(T - \xi, \varphi), \quad I^{1-\alpha}\varphi(0) = 0,
		\end{equation}
		where $\tilde{F} (s,\psi):= \frac{F(s,\psi)}{h(s)}$. Moreover, we have the following exponential-affine transform formula 
		\begin{align}\label{eq:laplace_fraction}
			\mathbb E \left[ \exp\left( \int_0^T f(s)^\top  X_s  \ud  s  \right) \right] = \exp\left( \phi(T) + X_0^\top\,I^{1-\alpha}\varphi\left(\int_0^T {h(u)\,\ud u}\right) \right),
		\end{align}
		where
		\begin{equation*}
			\phi(T) := \int_0^{\int_0^T {h(u)\,\ud u}} \frac 1{h(T-\xi(s))} {\varphi(s)^\top\,b^0\,\ud s}.
		\end{equation*}
\end{example}

\section{Proofs of Theorems~\ref{thm_weak_Rd} and \ref{thm_maininvariance}}\label{S:proofs_main_existence}
\subsection{A-priori estimates}

\begin{lemma}\label{lemma24}
    Let $T>0$. Let $\Gamma$ be a kernel satisfying Assumption~\ref{ass:conditions1_kernel} and $\eta>0$, $\gamma \in(0,1/2]$ denote the corresponding constants on~$[0,T]$. Let
    $(b_t)_{t\ge 0}$ and $(\sigma_t)_{t\ge 0}$ be càdlàg processes taking respectively their values in $\R^d$ and $\mathcal{M}_d(\R)$ such that $\sup_{t\in [0,T]} \E[|b_t|^p+|\sigma_t|^p]<\infty$ for some $p>1/\gamma$. Then, the process
$$Y_t=Y_0+\int_0^t \Gamma(t,s) b_s \ud s +\int_0^t \Gamma(t,s) \sigma_s \ud B_s $$
admits a version that is H\"older continuous on $[0,T]$ of any order $\alpha\in (0,\gamma-1/p)$, and this version satisfies 
$$\E\left[\left(\sup_{0\le s\le t \le T} \frac{|Y_t-Y_s|}{(t-s)^\alpha}\right)^p\right]\le c \sup_{t\in [0,T]} \E[|b_t|^p+|\sigma_t|^p],$$
where $c\in\R_+$ is a constant that only depends on $p$, $\eta$, $\gamma$ and $T$.    
\end{lemma}
\begin{proof}
The proof is a straightforward extension of the one of~\cite[Lemma 2.4]{AJLP19}  to double kernels: we first write 
\begin{align*}
    |Y_t-Y_s|^p\le& 4^{p-1}\left|\int_s^t \Gamma(t,u)b_u \ud u \right|^p + 4^{p-1}\left|\int_0^s (\Gamma(t,u)-\Gamma(s,u))b_u \ud u \right|^p \\
    & 4^{p-1}\left|\int_s^t \Gamma(t,u)\sigma_u \ud B_u \right|^p + 4^{p-1}\left|\int_0^s (\Gamma(t,u)-\Gamma(s,u))\sigma_u \ud B_u \right|^p,
\end{align*}
and use the same Jensen and BDG inequalities together with Assumption~\ref{ass:conditions1_kernel} to get
$$\E[|Y_t-Y_s|^p]\le c (t-s)^{p\gamma} \sup_{t\in [0,T]} \E[|b_t|^p+|\sigma_t|^p],$$
with a constant $c\in \R_+$ depending on $p$, $\eta$, $\gamma$ and $T$. The result follows from the Kolmogorov criterion~\cite[Theorem I.2.1]{ReYo}. 
\end{proof}

\begin{lemma}\label{lem:moment_bound}
	Let Assumptions \ref{ass:conditions1_coefficients} and \ref{ass:conditions1_kernel} hold. Let $X$ be a continuous solution of Equation~\eqref{eq:SVE}. Let $T>0$ and $\gamma \in(0,1/2],\eta>0$ the corresponding constants given by Assumption~\ref{ass:conditions1_kernel}.
    Then, for any $p>\frac{1}{\gamma}$, $\alpha\in (0,\gamma-\frac{1}{p})$, there exists $C\in\R_+$ depending only on $|X_0|$, $p$, $T$, $C_{LG}$, $\Gamma$, $\eta$ and $\alpha$ such that
\begin{equation}\label{eq:moment_bound}
	\E\Biggl[\sup_{t\in[0,T]}|X_t|^p +  \sup_{ 0\le s<t\le T}\frac{|X_t - X_s|^p}{|t-s|^{p\,\alpha}}\Biggr] \leq C.
	\end{equation}
    {In particular, $X$ admits on $[0,T]$ a modification with Hölder continuous paths of orders strictly less than $\gamma$.}
\end{lemma}
\begin{proof}
	Let $p\geq2$ and $T\in(0,+\infty)$. Since $X$ is continuous and adapted, $\tau_n := \inf\{t\geq 0 : |X_t| \geq n\}$ (with the usual convention $\inf \emptyset=+\infty$) is a stopping time for every $n\geq1$ and such that $\tau_n\to \infty$ almost surely as $n\to \infty$. Thus, a straightforward extension to double kernels of the proof of \cite[Lemma 3.1]{AJLP19} leads to the following inequality:
	\begin{equation*}
	1 + \E[\ind_{\{t < \tau_n\}}|X_t|^p] \leq C + C \,\int_0^t {\Gamma(t,s)^2\,\bigl(1 + \E[\ind_{\{s < \tau_n\}}|X_s|^p]\bigr)\,\ud s},
	\end{equation*}
	for all $t\in[0,T]$ with $C =\max(1+ 3^{p-1}|X_0|^p, C(p)C_{LG}^p  (\eta T^{2\gamma})^{\frac p 2 -1} )$, where $C(p) \in \R_+$ is a constant depending only on~$p$. 
    Here, we have used that $\sup_{t\in [0,T]} \int_0^t \Gamma(t,u)^2 \ud u \le \eta T^{2\gamma}$ by Assumption~\ref{ass:conditions1_kernel}. Still by Assumption \ref{ass:conditions1_kernel} (used this time for $T+1$ instead of~$T$), we have for all $\varepsilon\in(0,1]$,
	\begin{equation*}
	\sup_{t\in[0,T]}\,\int_t^{t+\varepsilon} {\Gamma(t+\varepsilon,s)^2\,\ud s} \leq \tilde{\eta}\,\varepsilon^{\tilde{\gamma}},
	\end{equation*}
	which tends to zero as $\varepsilon\to 0$. \cite[Lemma 2.1]{Zhang10} then gives the existence of the resolvent $R$ of $C\,\Gamma^2$ that satisfies for all $(t,s)\in\Delta_T$,
	$$R(t,s) - C\,\Gamma^2(t,s) = \int_s^t C\,\Gamma^2(t,u)\,R(u,s)\,\ud u = \int_s^t R(t,u)\,C\,\Gamma^2(u,s)\,\ud u.$$ The function $R:\Delta_T \to \R$ is measurable and such that $\sup_{t \in [0,T]} \left| \int_0^tR(t,s) \ud s \right| <\infty$
	Hence, by applying the Gr\"onwall-type inequality given in \cite[Lemma 2.2]{Zhang10}, observing that $t\mapsto1 + \E[\ind_{\{t < \tau_n\}}|X_t|^p]$ is by construction bounded on $[0,T]$, we get
	\begin{equation*}
	1 + \E[\ind_{\{t < \tau_n\}}|X_t|^p] \leq C +\int_0^t C R(t,s) \ud s \le C\left(1 +\sup_{t \in [0,T]} \left| \int_0^tR(t,s) \ud s \right|\right).
	\end{equation*}
	We then use Fatou's lemma to get the finiteness of $1 + \sup_{t\in[0,T]}\E[|X_t|^p]$, for any $p\ge 2$. Thereafter, using the bounds on the moments and Lemma~\ref{lemma24},
	leads to $\E\Bigl[ \sup_{0\le s<t\le T}\frac{|X_t - X_s|^p}{|t-s|^{p\,\alpha}}\Bigr] \leq C$ for any $p>\frac 1 \gamma$ and $\alpha\in (0,\gamma-\frac{1}{p})$. We finally write 
	\begin{equation*}
		|X_t|^p \leq 2^{p-1}\bigl(|X_0|^p+ |X_t - X_0|^p\bigr) \leq 2^{p-1}\Biggl(|X_0|^p+ T^{p\alpha} \frac{|X_t - X_0|^p}{t^{p \alpha }} \Biggr),
	\end{equation*}
	to get the bound on $\E\bigl[\sup_{t \in[0,T]} |X_t|^p\bigr]$.
\end{proof}

\subsection{Uniform estimates for the approximation scheme}\label{Subsec_scheme}
We consider the scheme for $(\hatX^N, \xi^N)$ of Section~\ref{S:scheme}.
Let $\eta(\cdot):[0,T]\to\{0,\cdots,N-1\}$ be such that $\eta(T) := N-1$ and for every $k\in\{0,\cdots,N-1\}$ and for all $t\in[t_k, t_{k+1})$, $\eta(t) := k$.  From~\eqref{eq:hatX}, we now rewrite $\hatX_t^N$ for $t\in[t_k, t_{k+1})$ as  
\begin{align}
\hatX_t^N &= X_0 + \sum_{j=1}^k \int_{t_{j-1}}^{t_j} {\!\Gamma(t, t_j)\bigl(b(\xi_{s}^N)\,\ud s + \sigma(\xi_{s}^N)\,\ud B^N_s\bigr)} \\
&= X_0 + \int_0^{t_{\eta(t)}} {\Gamma(t, t_{\eta(s)+1})\bigl(b(\xi_{s}^N)\,\ud s + \sigma(\xi_{s}^N)\,\ud B^N_s\bigr)}.\label{eq:hatX_bis}
\end{align}

Following Remark~\ref{rk_approx_scheme_Gamma}, we introduce a second scheme that does not require to have $\Gamma(t,t)\not=0$, but does not allow to exploit the nonnegativity preserving property. This scheme is defined,  for $t\in[t_0,t_1)$, by
$\check{X}^N_t=X_0$  and $\check{\xi}^N_t=X_0+\int_{t_0}^{t}(b(\check{\xi}^N_s) \ud s+ \sigma(\check{\xi}^N_s) \ud \check{B}^N_s)$. Then, for $t\in[t_1,t_2)$, we set   $\check{X}^N_{t}=X_0+(\check{\xi}^N_{t_1-}-X_0)\Gamma(t,t_1)=X_0+\Gamma(t,t_1)\int_{t_0}^{t_1}(b(\check{\xi}^N_s) \ud s+ \sigma(\check{\xi}^N_s) \ud \check{B}^N_s)$, and we consider $\check{\xi}^N_t$, a weak solution of 
$$\check{\xi}^N_t=\check{X}^N_{t_2-}+\int_{t_1}^{t}(b(\check{\xi}^N_s) \ud s+ \sigma(\check{\xi}^N_s) \ud \check{B}^N_s), \  t\in[t_1,t_{2}),$$
noting that $\check{X}^N_{t_2-}$ only depends on $(\check{B}^N_s,s\in [0,t_1))$. We construct then inductively the processes $\check{X}^N$ and $\check{\xi}^N$. Suppose that $(\check{X}^N_t,\check{\xi}^N_t)_{t\in [0,t_k)}$ is defined, we then  set 
\begin{equation}\label{def_check_X}
    \check{X}^N_{t}=X_0 +\sum_{j=1}^{k} \Gamma(t,t_j) \int_{t_{j-1}}^{t_j}(b(\check{\xi}^N_s) \ud s+ \sigma(\check{\xi}^N_s) \ud \check{B}^N_s), \ t\in[t_k,t_{k+1}).
\end{equation}
Besides there exists a weak solution $(\check{\xi}^N_t)_{t\in [t_k,t_{k+1})}$ and  a Brownian motion $\check{B}^N$ on $(t_k,t_{k+1})$ such that
\begin{equation}\label{def_check_xi}
    \check{\xi}^N_t=\check{X}^N_{t_{k+1}-}+\int_{t_k}^{t}(b(\check{\xi}^N_s) \ud s+ \sigma(\check{\xi}^N_s) \ud \check{B}^N_s), \  t\in[t_k,t_{k+1}).
\end{equation}
By construction, we have 
\begin{align}
\check{X}_t^N= X_0 + \int_0^{t_{\eta(t)}} {\Gamma(t, t_{\eta(s)+1})\bigl(b(\check{\xi}_{s}^N)\,\ud s + \sigma(\check{\xi}_{s}^N)\,\ud \check{B}^N_s\bigr)},\label{eq:checkX_bis}
\end{align}
which is analogous to~\eqref{eq:hatX_bis}. In contrast, we do not have in general $\check{X}^N_{t_k}=\check{\xi}^N_{t_k-}$, while we have $\hat{X}^N_{t_k}=\xi^N_{t_k-}$. This is the key property to deduce the invariance in a convex set~$\mathscr{C}$ from the scheme $\hat{X}^N$.

Let us denote by $w_{\Gamma,T}(\delta)$, for $\delta>0$, the modulus of continuity of $\Gamma$ over $\Delta_T$ given by
\begin{equation*}
w_{\Gamma,T}(\delta) := \max\bigl\{\bigl|\Gamma(t_1,s_1)-\Gamma(t_2,s_2)\bigr|:(t_1,s_1),(t_2,s_2)\in\Delta_T,\,|s_1-s_2|+|t_1-t_2|\leq\delta\bigr\}.
\end{equation*}
Then, we can establish the following uniform estimates on $(\check{X}^n, \check{\xi}^N)$ and $(\hatX^n, \xi^N)$. 
\begin{lemma}\label{lem:moment_bound_approximation}
	Let $\Gamma : \Delta_T \to \R_+$ be continuous.  Let  us assume that Assumption~\ref{ass:conditions1_coefficients} holds. 
	Let $p\ge 2$. Then, there exist constants $C_p,C \in \R_+$ depending only on $|X_0|$, $C_{LG}$ and $\max_{\Delta_T}\Gamma$ such that
\begin{equation*}
	\sup_{N \geq 1}\,\sup_{t\in[0,T]}\,\E\bigl[|\check{\xi}_t^N|^p + |\check{X}_t^N|^p\bigr] \leq C_p, \quad \sup_{t\in[0,T]}\E\bigl[|\check{\xi}_t^N - \check{X}_t^N|^2\bigr] \leq C\,\Biggl(\frac{T}{N} + w_{\Gamma,T}\biggl(\frac{T}{N}\biggr)^2\Biggr).
	\end{equation*}
    If we assume besides that $\Gamma(s,s)>0$ for all $s \in [ 0,T]$, there exist constants $C_p,C \in \R_+$ depending only on $|X_0|$, $C_{LG}$ and $\max_{\Delta_T}\Gamma$ such that
	\begin{equation*}
	\sup_{N \geq 1}\,\sup_{t\in[0,T]}\,\E\bigl[|\xi_t^N|^p + |\hatX_t^N|^p\bigr] \leq C_p, \quad \sup_{t\in[0,T]}\E\bigl[|\xi_t^N - \hatX_t^N|^2\bigr] \leq C\,\Biggl(\frac{T}{N} + w_{\Gamma,T}\biggl(\frac{T}{N}\biggr)^2\Biggr).
	\end{equation*}
\end{lemma}

\noindent The proof of Lemma~\ref{lem:moment_bound_approximation} is a straightforward extension  to double kernels and $p$-moments  of the one of \cite[Lemma 3.1]{Alfonsi23}.   Note that we also have $\sup_{t\in[0,T]}\E\bigl[|\xi_t^N - \hatX_t^N|^p\bigr] \leq C\,\Biggl(\left(\frac{T}{N}\right)^{p/2} + w_{\Gamma,T}\biggl(\frac{T}{N}\biggr)^p\Biggr)$ but we do not use this estimate in the paper.     {Let us note also that, under Lipschitz assumptions on $b$ and $\sigma$, we could get the strong rate of convergence $$\sup_{t\in[0,T]}\E\bigl[|X_t - \hatX_t^N|^p\bigr] \leq C\,\Biggl(\left(\frac{T}{N}\right)^{p/2} + w_{\Gamma,T}\biggl(\frac{T}{N}\biggr)^p\Biggr),$$
in the same way as \cite[Proposition 3.2]{Alfonsi23}. }

\subsection{Weak convergence of the scheme for continuous kernels.}

We start by showing weak existence  for continuous kernels $\Gamma$ by weak convergence of the scheme of Section~\ref{S:scheme}.

\begin{lemma}\label{lemma_weak_C0}
	Let $T>0$.  Let $\Gamma:\Delta_T \to \R$ be continuous and  satisfying Assumption~\ref{ass:conditions1_kernel}.  Let $b,\sigma$ satisfy Assumption~\ref{ass:conditions1_coefficients}.  Then, there exists a continuous weak solution to the SVE \eqref{eq:SVE} for any $X_0 \in \R^d$.
	If moreover $\mathscr{C}\subset \R^d$ is closed and convex, $X_0 \in \mathscr{C}$, $\Gamma$ satisfies Assumption~\ref{ass:preserving_nonnegativity} and~\eqref{WSI_lambda} holds for any $\lambda\in \{ \Gamma(t,t), t \in [0,T]\}$,  then, there exists a weak solution to the SVE that is continuous and stays in $\mathscr{C}$, i.e. $\P(X_t \in \mathscr{C}, t\in [0,T])=1$.
\end{lemma}

\begin{proof}
	We start by proving the result when $\Gamma$ satisfies Assumption~\ref{ass:preserving_nonnegativity}. We first define exactly as in Section~\ref{S:scheme} the processes $(\hatX^N_t)_{t\in [0,T]}$ and $(\xi^N_t)_{t\in [0,T]}$, associated to a Brownian motion $B^N$ on a filtered probability space $(\Omega^N,\cF^N,(\cF^N_t)_{t\in [0,T]},\P^N)$. 
 Let us also define the process 
\begin{equation}\label{def_XtildeN}\tilde{X}^N_t=X_0+\int_0^t \Gamma(t,s) \left(b(\xi^N_s) \ud s+ \sigma(\xi^N_s) \ud B^N_s \right),\ t\in[0,T].\end{equation}
We have 
\begin{align*}
	\hatX^N_t-\tilde{X}^N_t=&\int_0^{\eta(t)}[\Gamma(t,t_{\eta(s)+1})-\Gamma(t,s)][b(\xi^N_s) \ud s + \sigma (\xi^N_s) \ud B^N_s ] \\
	&-\int_{\eta(t)}^t \Gamma(t,s)[b(\xi^N_s) \ud s + \sigma (\xi^N_s) \ud B^N_s ].
\end{align*} By using Jensen inequality, Itô isometry, Lemma~\ref{lem:moment_bound_approximation} and the continuity of $\Gamma$ over $\Delta_T$, we get that there is a constant $C$ depending only on $X_0$, $C_{LG}$, $T$ and $\max_{0\le s\le t\le T} \Gamma(t,s)$ such that for all $N\ge 1$,
	$$\forall t \in[0,T],\ \E[|\hatX^N_t-\tilde{X}^N_t|^2] \le C\left(\frac T N +w_{\Gamma,T}(T/N)^2 \right),$$
	which gives\begin{equation}\label{cv_Xtilde}
		\forall t \in[0,T],\ \E[|\xi^N_t-\tilde{X}^N_t|^2] \le C\left(\frac T N +w_{\Gamma,T}(T/N)^2 \right)
	\end{equation}
	by Lemma~\ref{lem:moment_bound_approximation}.

	By using again Lemma~\ref{lem:moment_bound_approximation},  $\xi^N$ has bounded moments, i.e. $\sup_{t\in [0,T]}\E[|\xi^N_t|^p]\le C_p$	 and then by Lemma~\ref{lemma24}, we get that $\tilde{X}^N$ and thus $(\tilde{X}^N, B^N)$ satisfies the Kolmogorov-Centsov criterion that gives the $C$-tightness, see e.g.~\cite[Problem 2.4.11]{KS}.  We consider a converging subsequence that we still denote $(\tilde{X}^N,B^N)$ to lighten notation and we note $(X,B)$ its limit.
	
Let $Q\in \mathbb{N}^*$, $0\le t_1<\dots< t_Q\le T$ and, for $q\in \{1,\dots,Q\}$, $t\in [0,T]$,
\begin{eqnarray*}
\tilde{Z}^{q,N}_t&=&X_0+\int_0^t \Gamma(t_q,s) [b(\tilde{X}^N_s) \ud s + \sigma (\tilde{X}^N_s) \ud B^N_s]. \\
Z^{q}_t&=&X_0+\int_0^t \Gamma(t_q,s) [b(X_s) \ud s + \sigma (X_s) \ud B_s],
\end{eqnarray*}
with the convention that $\Gamma(t,s)=0$ if $s>t$. By using~\cite[Theorem 7.10]{KP} (observe that $B^N$ satisfies trivially the uniform tighness (UT) assumption of~\cite[Definition 7.4]{KP}), we get that 
$$(\tilde{X}^N,B^N,\tilde{Z}^{1,N},\dots,\tilde{Z}^{Q,N})\implies (X,B,Z^1,\dots,Z^Q),$$
and then in particular\footnote{Note that we use the weak convergence of processes given by~\cite{KP} only to get the weak convergence of some stochastic integrals. However, we have not found a reference for this weaker result.}
\begin{equation}\label{eq_weak1}(\tilde{X}^N,B^N,\tilde{Z}^{1,N}_{t_1},\dots,\tilde{Z}^{Q,N}_{t_Q})\implies (X,B,Z^1_{t_1},\dots,Z^Q_{t_Q}).
\end{equation}

We now define 
$$Z^{q,N}_t=X_0+\int_0^t \Gamma(t_q,s) [b(\xi^N_s) \ud s + \sigma (\xi^N_s) \ud B^N_s].$$
From~\eqref{cv_Xtilde}, we get $\int_0^T \E[|\tilde{X}^N_s-\xi^N_s|^2]ds\to 0$ and we may then assume w.l.o.g. (considering a subsequence) that $\tilde{X}^N_s-\xi^N_s\to 0$ $ds \otimes \P$ almost everywhere. We prove the convergence of $\P( |\tilde{Z}^{q,N}_{t_q}-{Z}^{q,N}_{t_q}|>\epsilon) \to 0$ as $N\to \infty$, for $\epsilon>0$. 
In the particular case where $b$ and $\sigma$ are bounded, we deduce the convergence in probability from Markov inequality, Jensen inequality, Itô isometry and the continuity of $\Gamma$ over $\Delta_T$,
$$\P( |\tilde{Z}^{q,N}_{t_q}-{Z}^{q,N}_{t_q}|>\epsilon)\le \frac 1 {\epsilon^2}\E[|\tilde{Z}^{q,N}_{t_q}-{Z}^{q,N}_{t_q}|^2]\le \frac{2}{\epsilon^2} \E \left[\int_0^{t_q} T(b(\tilde{X}^N_s)-b(\xi^N_s))^2+(\sigma(\tilde{X}^N_s)-\sigma(\xi^N_s))^2 ds \right],$$
and conclude with the dominated convergence theorem. 

In the general case, we know from Lemmas~\ref{lem:moment_bound_approximation} and~\ref{lemma24} that there exists $C\in \R_+$ such that $$\sup_{N\ge 1}\E\left[\sup_{s\in[0,T]}|\tilde{X}^N_s|^2 \right]\le C.$$ 
For $\gamma>0$, we introduce $\tau^{N,\gamma}=\inf \{ t\in[0,T]: |\tilde{X}^N_t|>\gamma \}$ (convention $\inf\emptyset= +\infty$), $b_\gamma(x)=b(\pi_{\gamma}(x))$ , $\sigma_\gamma(x)=b(\pi_{\gamma}(x))$, where $\pi_\gamma(x)=\frac{\gamma}{\gamma\vee |x|}x$ is the projection on the closed ball of radius~$\gamma$. We have $\P(\tau^{N,\gamma}\le T)\le \P(\sup_{s\in[0,T]}|\tilde{X}^N_s|>\gamma) \le \frac{C}{\gamma^2}$. From $\mathbf{1}_{|\tilde{Z}^{q,N}_{t_q}-{Z}^{q,N}_{t_q}|>\epsilon}\le \mathbf{1}_{\tau^{N,\gamma}\le T}+\mathbf{1}_{\tau^{N,\gamma}> T} \frac{|\tilde{Z}^{q,N}_{t_q}-{Z}^{q,N}_{t_q}|^2}{\epsilon^2}$, where we have used that $\mathbf{1}_{|x|>\epsilon}\leq\frac{|x|}{\epsilon}$, we get
$$\P( |\tilde{Z}^{q,N}_{t_q}-{Z}^{q,N}_{t_q}|>\epsilon)\le \frac{C}{\gamma^2}+\frac{1}{\epsilon^2}\E\left[\mathbf{1}_{\tau^{N,\gamma} > T} |\tilde{Z}^{q,N}_{t_q}-{Z}^{q,N}_{t_q}|^2\right].$$
We analyse the second term and have
\begin{align*}
	&\E\left[\mathbf{1}_{\tau^{N,\gamma} > T} |\tilde{Z}^{q,N}_{t_q}-{Z}^{q,N}_{t_q}|^2\right] \\
	&=\E\left[\mathbf{1}_{\tau^{N,\gamma} > T} \left|\int_0^{t_q} \Gamma(t_q,s) [ (b_\gamma(\tilde{X}^N_s) -b(\xi^N_s) ) \ud s + (\sigma_\gamma(\tilde{X}^N_s)-\sigma (\xi^N_s)) \ud B^N_s ] \right|^2\right]\\
	&\le 2 	\E\left[ \left|\int_0^{t_q} \Gamma(t_q,s) [ (b_\gamma(\tilde{X}^N_s) -b_\gamma(\xi^N_s) ) \ud s + (\sigma_\gamma(\tilde{X}^N_s)-\sigma_\gamma (\xi^N_s)) \ud B^N_s ] \right|^2 \right] \\
	&\ + 2 	\E\left[ \left|\int_0^{t_q} \Gamma(t_q,s) [ (b_\gamma(\xi^N_s)-b(\xi^N_s) ) \ud s + (\sigma_\gamma (\xi^N_s)-\sigma(\xi^N_s)) \ud B^N_s ] \right|^2 \right]\\
	&\le 4 \E \left[\int_0^{t_q} T|b_\gamma(\tilde{X}^N_s)-b_\gamma(\xi^N_s)|^2+|\sigma_\gamma(\tilde{X}^N_s)-\sigma_\gamma(\xi^N_s)|^2 ds \right] \\
	&\ +  4 \E \left[\int_0^{t_q} T|b_\gamma(\xi^N_s)-b(\xi^N_s)|^2+|\sigma_\gamma(\xi^N_s)-\sigma(\xi^N_s)|^2 ds \right],
\end{align*}
by using Jensen inequality, It\^o isometry and the continuity of $\Gamma$ over $\Delta_T$. The first term goes to zero as $N\to \infty$ by dominated convergence theorem as $b_\gamma$ and $\sigma_\gamma$ are bounded continuous. Since $|b_\gamma(x)-b(x)|+|\sigma_\gamma(x)-\sigma(x)|\le 2C_{LG}(1+|x|)\mathbf{1}_{|x|>\gamma}$, the second term is upper bounded by 
$$4C_{LG}^2  \int_0^T \E[(1+|\xi^N_s|^2)\mathbf{1}_{|\xi^N_s|>\gamma}] ds \le 4C_{LG}^2 C_3 T \left(\gamma^{-3}+\gamma^{-1}\right),$$
by using Lemma~\ref{lem:moment_bound_approximation} with $p=3$ and again $\mathbf{1}_{|x|>\epsilon}\leq\frac{|x|}{\epsilon}$. Therefore,  $\limsup_{N\to \infty} \P( |\tilde{Z}^{q,N}_{t_q}-{Z}^{q,N}_{t_q}|>\epsilon) \le \frac{C}{\gamma^2}+ \frac{4C_{LG}^2 C_3 T}{\epsilon^2}\left(\gamma^{-3}+\gamma^{-1}\right)$, which gives the desired convergence in probability since $\gamma$ can be arbitrary large. 

Since $\tilde{Z}^{q,N}_{t_q}-{Z}^{q,N}_{t_q}\to 0$ in probability for all $q\in \{1,\dots,Q\}$, we then get from~\eqref{eq_weak1}
$$(\tilde{X}^N,B^N,Z^{1,N}_{t_1},\dots,Z^{Q,N}_{t_Q})\implies (X,B,Z^1_{t_1},\dots,Z^Q_{t_Q}), $$
and in particular
$$ (\tilde{X}^N_{t_1}-Z^{1,N}_{t_1},\dots,\tilde{X}^N_{t_Q}-Z^{Q,N}_{t_Q})\implies (X_{t_1}-Z^{1}_{t_1},\dots,X_{t_Q}-Z^{Q}_{t_Q}).$$
This gives that $(X_{t_1}-Z^{1}_{t_1},\dots,X_{t_Q}-Z^{Q}_{t_Q})=0$, a.s. However, the processes $(X_t-X_0)_{t\in[0,T]}$ and $\left(\int_0^t \Gamma(t,s)[b(X_s) \ud s+\sigma(X_s)\ud W_s]\right)_{t\in[0,T]}$ are continuous (from the $C$-tightness for the first one and using Lemma~\ref{lemma24} for the second one using the uniform bounds on the moments of~$X$), they therefore coincides for every $t\in[0,T]$. This shows the first claim.

We now prove the second part of the claim. By using Lemma~\ref{lem:nonnegative_approx}, we have $\P(\xi^N_t \in \mathscr{C}, t \in[0,T])=1$. From~\eqref{cv_Xtilde}, it comes that $X_t\in \mathscr{C}$ for any $t\in [0,T]$ and thus $\P(X_t \in \mathscr{C}, t \in[0,T])=1$ since $X$ is continuous and $\mathscr{C}$ is a closed set.

Last, we prove the existence result without Assumption~\ref{ass:preserving_nonnegativity}.
In this case, we work with the approximation scheme $\check{X}^N$ and $\check{\xi}^N$ by~\eqref{def_check_X} and~\eqref{def_check_xi}. We then set $$\tilde{X}^N_t=X_0+\int_0^t \Gamma(t,s) \left(b(\check{\xi}^N_s) \ud s+ \sigma(\check{\xi}^N_s) \ud \check{B}^N_s \right),\ t\in[0,T],$$
instead of~\eqref{def_XtildeN}. We get in the same way $\E[|\check{X}^N_t-\tilde{X}^N_t|^2] \le C\left(\frac T N +w_{\Gamma,T}(T/N)^2 \right)$ for $t\in[0,T]$ and then by Lemma~\ref{lem:moment_bound_approximation} $$\forall t \in[0,T],\ \E[|\check{\xi}^N_t-\tilde{X}^N_t|^2] \le C\left(\frac T N +w_{\Gamma,T}(T/N)^2 \right).$$
We repeat then exactly the same arguments to prove the existence of a weak solution. 
 \end{proof}

\subsection{Passing to more general kernels.}

The final ingredient needed is a stability result to allow for  possibly singular kernels. 

\begin{lemma}\label{lemma_stability}
    Let $T>0$ and $X_0 \in \R^d$.  Fix $\Gamma$ satisfying Assumption~\ref{ass:conditions1_kernel},
 and let $(\Gamma_M)_{M \in \mathbb N}$ be a sequence of kernels satisfying $\int_0^t(\Gamma(t,s)-\Gamma_M(t,s))^2 \ud s \to_{M\to \infty} 0$  for all $t\ge 0$, and such that for every $T\in(0,+\infty)$, there exist $\eta,\gamma>0$ such that for all $M\in\N$,
\begin{equation*}
	\int_s^t {\Gamma_M(t,u)^2\,\ud u} + \int_0^s {(\Gamma_M(t,u) - \Gamma_M(s,u))^2\,\ud u} \leq \eta\,|t-s|^{2\gamma}, \quad (t,s)\in\Delta_T.
	\end{equation*}
  Let $b,\sigma$ satisfy Assumption~\ref{ass:conditions1_coefficients} and assume the existence of a sequence of  continuous weak solution $(X^M)_{M \in \mathbb N}$ to the SVEs
 \begin{align}\label{eq:SVE_M}
     X^M_t=X_0+\int_0^t {\Gamma_M(t,s)\Bigl(b(X^M_s) \ud s + \sigma (X^M_s) \ud B^M_s\Bigr)}.
 \end{align}
 Then for every~$T>0$, the sequence $(X^M, B^M)_{M \in \mathbb N}$ is tight for the uniform topology on $[0,T]$ and any limiting point $(X,B)$ is a continuous weak solution to the SVE \eqref{eq:SVE_M} with the kernel~$\Gamma$. 
\end{lemma}
Let us point here that when $\Gamma_M$ is continuous, which is the case when using Lemma~\ref{lem_approx_kernel} or Assumption~\ref{ass:conditions_kernel_limweak}, then the weak existence of~\eqref{eq:SVE_M}
follows from Lemma~\ref{lemma_weak_C0}.

\begin{proof}
	We first prove the Kolmogorov-Centsov criterion. Note that we cannot apply Lemmata \ref{lemma24} and \ref{lem:moment_bound} here as the aimed upper bound must be independent of $M$ to get the tightness of the sequence $(X^M, B^M)_{M \in \mathbb N}$. Let $p\ge 2$ and $\tau^{M,N}=\inf\{t\ge 0: |X^M_t|\ge N\}$ for $N\ge 1$. We have 
\begin{align*}|X^M_t-X^M_s|^p\mathbf{1}_{\tau^{M,N}>t}\le &\left|\int_s^t\Gamma_M(t,u)\mathbf{1}_{\tau^{M,N}>u}[b(X^M_u) \ud u + \sigma (X^M_u) \ud B^M_u] \right|^p\\&+\left|\int_0^s [\Gamma_M(t,u)-\Gamma_M(s,u)]\mathbf{1}_{\tau^{M,N}>u}[b(X^M_u) \ud u + \sigma (X^M_u) \ud B^M_u] \right|^p, \end{align*}
and with the same arguments as in the proof Lemma~\ref{lemma24}, we get that there exists a constant $c'$ that only depends on $p$, $\eta$, $\gamma$, $C_{LG}$ and $T$ such that for $0\le s<t\le T$,
$$\E[|X^M_t-X^M_s|^p\mathbf{1}_{\tau^{M,N}>t}]\le c' \left(1+ \sup_{u\le t} \E[|X^M_u|^p\mathbf{1}_{\tau^{M,N}>u}]\right) (t-s)^{\gamma p}.$$
We first prove that \begin{equation}\label{borne_unif}\sup_{u\le T} \E[|X^M_u|^p]<C<\infty\end{equation} for a constant~$C$ that only depends on $p$, $\eta$, $\gamma$, $C_{LG}$, $T$ and $X_0$. To do so, we  take $s=0$ and $\tilde{T}=\min(T,T_0)$ with $c'2^{p-1}T_0^{\gamma p}=1/2$. We get for $t\in[0,\tilde{T}]$
$$ \E[|X^M_t-X_0|^p\mathbf{1}_{\tau^{M,N}>t}]\le c' \left(1+ 2^{p-1}\E[|X_0|^p] \right) T_0^{\gamma p} + \frac 1 2 \sup_{u\le t} \E[|X^M_u-X_0|^p\mathbf{1}_{\tau^{M,N}>u}].$$
Taking the supremum on $t\in [0,\tilde{T}]$, we get  $\sup_{t\le \tilde{T}} \E[|X^M_u-X_0|^p\mathbf{1}_{\tau^{M,N}>u}]\le 2c' \left(1+ 2^{p-1}\E[|X_0|^p] \right) T_0^{\gamma p}$. Letting $N\to \infty$, we get $\sup_{t\le \tilde{T}} \E[|X^M_u-X_0|^p]\le 2c' \left(1+ 2^{p-1}\E[|X_0|^p] \right) T_0^{\gamma p}$. If $\tilde{T}=T$, the claim is proved. Otherwise, we repeat the argument $\lceil T/T_0 \rceil$ times on the intervals $[i T_0,(i+1)T_0 \wedge T]$ for $i=1,\dots, \lceil T/T_0 \rceil-1$, to get the uniform bound on $\sup_{u\le T} \E[|X^M_u|^p]$. Then, Lemma~\ref{lemma24} gives 
$$\E\left[\left(\sup_{0\le s<t\le T} \frac{|X^M_t-X^M_s|}{|t-s|^\alpha}\right)^p\right]\le C,$$
for all $\alpha \in [0,\gamma-1/p)$, where $C$ is a constant depending only on  $p$, $\eta$, $\gamma$, $C_{LG}$, $T$, $X_0$ and $\alpha$. It does not depend on~$M$ and gives the $C$-tightness of $(X^M,B^M)$ by the Kolmogorov criterion.

 There is thus a converging subsequence that we still denote $(X^M,B^M)\implies (X,B)$. We show that $X$ is a weak solution to the SVE similarly as for Lemma~\ref{lemma_weak_C0}.  Namely, for $Q\in \mathbb{N}^*$, $0\le t_1<\dots<t_Q\le T$, we get by~\cite[Theorem 7.10]{KP} that 
	$$(X^M,B^M,\tilde{Z}^{1,M},\dots,\tilde{Z}^{Q,M})\implies (X,B,Z^1,\dots,Z^Q),$$
	where $\tilde{Z}^{q,M}_t=\int_0^t \Gamma(t_q,s)[b(X^M_s) \ud s + \sigma(X^M_s) \ud B^M_s]$ and $\tilde{Z}^{q}_t=\int_0^t \Gamma(t_q,s)[b(X_s) \ud s + \sigma(X_s) \ud B_s]$, for $t\in [0,T]$ and $q\in \{1,\dots,Q\}$. We get in particular 
	$$(X^M,B^M,\tilde{Z}^{1,M}_{t_1},\dots,\tilde{Z}^{Q,M}_{t_Q})\implies (X,B,Z^1_{t_1},\dots,Z^Q_{t_Q}).$$
	We now define  
	$$Z^{q,M}_t=\int_0^t \Gamma_M(t_q,s)[b(X^M_s) \ud s + \sigma(X^M_s) \ud B^M_s], \quad t \in [0,T].$$ 
	We have bounds on second moment by~\eqref{borne_unif} with $p=2$  and therefore we get by Jensen inequality and Itô isometry 
	\begin{align*}
		\E[|Z^{q,M}_{t_q}-\tilde{Z}^{q,M}_{t_q}|^2] \le  \int_0^{t_q} 	(\Gamma(t_q,s)-\Gamma_M(t_q,s))^2 (T\E[|b(X^M_s)|^2]+\E[|\sigma(X^M_s)|^2])\ud s \to 0,
	\end{align*}
	by using~\eqref{borne_unif}, Assumption~\ref{ass:conditions1_coefficients} and knowing that $\int_0^t(\Gamma(t,s)-\Gamma_M(t,s))^2 \ud s \to_{M\to \infty} 0$  for all $t\ge 0$. We then conclude as for the previous theorem: we have $(X^M,B^M,{Z}^{1,M}_{t_1},\dots,{Z}^{Q,M}_{t_Q})\implies (X,B,Z^1_{t_1},\dots,Z^Q_{t_Q})$, which gives $X_{t_q}=X_0+\int_0^{t_q}\Gamma(t_q,s)[b(X_s) \ud s+\sigma(X_s)\ud B_s]$ for $1\le q\le Q$. Using then the continuity of the processes $X$ and $\left( \int_0^t \Gamma(t,s) [b(X_s) \ud s+\sigma(X_s)\ud B_s] \right)_{t\in[0,T]}$ (by Lemma~\ref{lemma24}), we get that $X$ solves the SVE~\eqref{eq:SVE}.
\end{proof}

\subsection{Putting everything together}
Using Lemmas~\ref{lemma_weak_C0} and \ref{lemma_stability}, we can now prove Theorems~\ref{thm_weak_Rd} and \ref{thm_maininvariance} as follows. 

\begin{proof}[Proof of Theorem~\ref{thm_weak_Rd}]  Let $X_0 \in \R^d$. Since $\Gamma$ satisfies Assumption \ref{ass:conditions1_kernel}, there exists by Lemma~\ref{lem_approx_kernel}  a sequence of continuous approximating kernels $\Gamma_M$ that satisfy the required assumption for the stability Lemma~\ref{lemma_stability}. By Lemma~\ref{lemma_weak_C0}, we know that there exists a weak solution to the SVE~\eqref{eq:SVE_M}, and we apply Lemma~\ref{lemma_stability} to get the weak existence of~$X$. The H\"older continuity follows from Lemma~\ref{lem:moment_bound}.
\end{proof}
\begin{proof}[Proof of Theorem~\ref{thm_maininvariance}]
Let $T>0$. If $\Gamma$ is continuous on $\Delta_T$ and satisfies Assumption \ref{ass:preserving_nonnegativity}, then we simply apply directly the second part of Lemma~\ref{lemma_weak_C0}. Otherwise, $\Lambda_T=\R_+$, and the proof follows the same arguments as the one of Theorem~\ref{thm_weak_Rd}, but we use the sequence of kernels~$\Gamma_M$ given  by Assumption~\ref{ass:conditions_kernel_limweak}. These kernels satisfy Assumption \ref{ass:preserving_nonnegativity}, they in particular preserves nonnegativity and are positive on the diagonal, which enables us to use the second part of Lemma~\ref{lemma_weak_C0}.
 This gives a weak solution to the SVE~\eqref{eq:SVE_M} that stays in~$\mathscr{C}$ for all $t\in[0,T]$. We then apply Lemma~\ref{lemma_stability} to get the weak existence of a continuous solution~$X$. Since $\mathscr{C}$ is closed, we also deduce from the weak convergence of the marginal laws that $\P(X_t\in \mathscr{C})=1$ for all $t\in [0,T]$ and then $\P(\forall t \in [0,T], \ X_t\in \mathscr{C})=1$ by using the continuity of~$X$. 
\end{proof}

\section{Proof of Theorem~\ref{T:VolSqrt}}\label{S:proofsquareroot}

\subsection{Existence for the Stochastic Volterra equation}\label{S:existenceRd+}
We first argue the existence of an $\R^d_+$-valued solution $X$ to the stochastic Volterra equation \eqref{eq:VolSqrt}. For this, we define the coefficients $b:\R^d \to \R^d$ and $\sigma: \R^d \to \mathcal M_d(\mathbb R)$ by 
\begin{align}
    b(x) := b^0 + B x,  \quad \sigma(x)  = \mbox{diag}\left( \sigma_1 \sqrt{x_1^+}, \ldots, \sigma_d \sqrt{x_d^+}\right), \quad x \in \R^d, 
\end{align}
where $y^+ := \max(0,y)$. Clearly, the coefficients $b$ and $\sigma$ are continuous with at most linear growth in the sense of Assumption~\ref{ass:conditions1_coefficients}   and satisfy the  conditions \eqref{eq:bsigmaRd+} thanks to the structural assumptions on $b^0$ and $B$ in \eqref{sqrt2}. Hence, Assumption \eqref{WSI_lambda} holds for any $\lambda>0$, for $\mathscr C = \R^d_+$ so that an application of Theorem~\ref{thm_maininvariance} yields the existence of an $\R^d_+$-valued continuous solution $X$ to the stochastic Volterra equation \eqref{eq:VolSqrt} for any $X_0 \in \R^d_+$. The Hölder regularity of the sample paths follows from Lemma~\ref{lem:moment_bound}.

To argue uniqueness, we start by deriving the exponential-affine transform formula in \eqref{eq:laplace}. This is the aim of  the following two sections.

\subsection{A verification result}

\begin{lemma}\label{L:verification} Fix a kernel $\Gamma : \Delta_T \to \R_+$ satisfying  Assumptions \ref{ass:conditions1_kernel} and let $X$ be an $\R^d_+$-valued solution to~\eqref{eq:VolSqrt}. Fix $f\in C([0,T],\R^d)$.  Assume there exists a solution $\psi \in C([0,T], \R^d)$ to the Riccati-Volterra equation \eqref{RicVolSqrt}. 
Then, the expression \eqref{eq:laplace} for the Laplace transform holds, 
for all $t\in [0,T]$.
\end{lemma}

We provide a brief outline of the proof, since the same strategy has been used in \cite*[Theorem 4.3]{AJLP19} in the convolution setting and   in \cite*[Theorem 2.1]{AKO22} for the non-convolution setting, in terms of the forward process $\mathbb E[X_s | \mathcal F_t]$, which is different from $g_t(s)$ in \eqref{eq:gts}. We note here that Assumption~\ref{ass:preserving_nonnegativity} is not needed on the kernel $\Gamma$. 

\begin{proof}[Proof of Lemma~\ref{L:verification}.]
Define 
$$ U_t = \int_0^t f(s)^\top X_s  \ud 
 s +  \int_t^T F(s, \psi(s))^\top  g_t(s) \ud s $$  
and set $ M = \exp(U)$. To obtain \eqref{eq:laplace}, it suffices to prove that $M$ is a martingale. Indeed, if this the case then, the martingale property yields
\begin{align}
    \mathbb E\left[ \exp\left( 
 \int_0^T f(s)^\top X_s  \ud  s  \right)  \Big | \mathcal F_t \right] &= \mathbb E\left[ M_T  \Big | \mathcal F_t \right]\\ &= M_t \\&= \exp\left(\int_0^t f(s)^\top X_s  \ud  s +  \int_t^T F(s, \psi(s))^\top  g_t(s) \ud  s\right),
\end{align}
which yields \eqref{eq:laplace}. We now argue martingality of $M$ by computing its dynamics using Itô's formula:
\begin{align}\label{eq:MIto}
    \frac{ \ud  M_t}{M_t} =  dU_t + \frac 1 2 d\langle U\rangle_t. 
\end{align}
The dynamics of $U$ can readily be obtained by recalling $g_t(s)$ from  \eqref{eq:gts} and by observing that for fixed $s$, the dynamics of $t\to g_t(s)$ are given by 
$$ \ud g_t(s) = \Gamma(s,t)  \ud  Z_t \quad t\leq s.  $$
Since  $g_t(t) = X_t$, it follows that 
\begin{align}
     \ud  U_t &= \left(f(t)^\top X_t  - F(t,\psi(t))^\top X_t \right)  \ud t  + \int_t^T F(s,\psi(s))^\top \Gamma(s,t)  \ud  s  \ud  Z_t \\
    &= \left(f(t)^\top X_t  -F(t,\psi(t))^\top X_t \right)  \ud t  + \psi(t)^\top   \ud  Z_t, 
    \end{align}
    where for the second equality we used the Riccati--Volterra equation \eqref{RicVolSqrt}. This implies that 
    $$ d\langle U\rangle_t = \sum_{i=1}^d \psi_i^2(t) \sigma_i^2 X_{i,t}  \ud  t .$$
    Injecting the dynamics of $dU$ and $d\langle U \rangle$ in \eqref{eq:MIto}, we get that 
    \begin{align}
        \frac{ \ud  M_t}{M_t} &= \sum_{i=1}^d\left( f_i(t) - F_i(t,\psi(t))  + (B^\top \psi(t))_i + \frac {\sigma_i^2 } 2  \psi_i^2(t) \right) X_{i,t} \ud t +  \sum_{i=1}^d \sigma_i\psi_i(t)\sqrt{X_{i,t}} \ud W_{i,t}, \\
        &=  \sum_{i=1}^d \sigma_i\psi_i(t)\sqrt{X_{i,t}} \ud W_{i,t},
    \end{align}
    where the drift vanishes in the second equality by definition of $F$ in \eqref{RicVolSqrt}.  This shows that $M$ is an exponential  local martingale of the form 
    \begin{align}
        M_t = M_0 \exp\left(  \sum_{i=1}^d \int_0^t  \sigma_i\psi_i(s)\sqrt{X_{i,s}} \ud W_{i,s}  - \frac 1 2  \sum_{i=1}^d \int_0^t  \sigma_i^2\psi_i^2(s) {X_{i,s}} \ud  s  \right).
    \end{align}
    The martingality of $M$ is obtained from a straightforward adaptation of \cite[Lemma 7.3]{AJLP19} to the non-convolution and multi-dimensional setting since $\psi$ is real-valued and continuous and hence bounded on $[0,T]$. {Namely, let us define the sequence of stopping times $\tau_n=\inf \{t\ge 0 : \max_{1\le i \le d} X_{i,t} \ge n \}$ and $\frac{d \mathbb{Q}^n}{d \P}=M_{\tau_n \wedge T}$. Then  by Girsanov's theorem, $\ud W^n_{i,s}= \ud W_{i,s} - \mathbf{1}_{[0,\tau_n]}(s) \sigma_i \psi_i(s) \sqrt{X_{i,s} }\ud s$ is a Brownian motion under 
    $\mathbb{Q}^n$ and we have 
    \[ X_{i,t} = X_{0,i} + \int_0^t \Gamma(t,s) \left(b^0+ (BX_s)_i +\mathbf{1}_{[0,\tau_n]}(s)\sigma^2_i\psi_i(s){X_{i,s}} \right) \ud s
+ \int_0^t \Gamma(t,s) \sigma_i\psi_i(s)\sqrt{X_{i,s}} \ud 
 W^n_{i,s},\]  
 Since $\psi$ is continuous, it is bounded on $[0,T]$. The drift and the volatility coefficients are thus upper bounded by $C_{LG}(1+|X_s|)$ uniformly in~$n$, so that a straightforward adaptation of Lemma~\ref{lem:moment_bound} to time dependent coefficients gives  $\E^{\mathbb{Q}^n}\left( \sup_{t\in[0,T]} |X_t|^2\right)\le C$, for a constant $C$ that does not depend on~$n$. Therefore, $\mathbb{Q}^n(\tau_n<T)= \mathbb{Q}^n(\sup_{t\in[0,T]} \max_{1\le i \le d} X_{i,t} \ge n)=O(1/n^2)$, and we have $\E[M_T]\ge \E[M_T \mathbf{1}_{\tau_n\ge T}]=\mathbb{Q}^n(\tau_n\ge T)\to_{n\to \infty}1$. Since $M$ is clearly a supermartingale, we get $\E[M_T]=1$ and that $M$ is thus a martingale. 
 }
\end{proof}

\subsection{Existence for the Riccati-Volterra equation}

 The existence of a solution to the Riccati-Volterra equation \eqref{RicVolSqrt} is obtained in  Lemma~\ref{L:Riccatiexistence} below. It relies on two elementary results. 

 The first one concerns the existence of local solutions to deterministic Volterra equations: 
\begin{lemma}\label{L:Riccatilocalexistence} Fix a kernel $\Gamma$ satisfying  Assumption~\ref{ass:conditions1_kernel}. Define the function $p: \R_+ \times \R^d \to \R^d$ by
$$ p_i(t,x):=x^\top A_i(t)x + b_i(t)^\top x + c_i(t), \quad i=1\ldots, d, \quad (t,x)\in \R_+\times \R^d,$$ where $A_i:\R_+\to \mathcal M_d(\R)$, $b_i:\R_+\to \R^d$, $c_i$ are continuous functions. Let  $g \in \mathcal C(\R_+,\R^d)$.    Then, the Volterra equation 
\begin{align}\label{eq:localexistence}
        \psi(t) = g(t) + \int_0^t \Gamma(t,s) p(s,\psi(s)) \ud s
    \end{align}
admits a unique  non-continuable  solution $\psi \in C([0,T_{max}),\R^d)$ in the sense that $\psi$ satisfies   \eqref{eq:localexistence} on $[0,T_{max})$ with $T_{max} \in (0,\infty]$ and $\sup_{t<T_{max}}|\psi( t)| = +\infty$, if $T_{max}<\infty$. 
\end{lemma}

\begin{proof}
This follows  from \cite[Theorem 12.2.6]{GLS90}.
\end{proof}

 The second result deals with non-negativity of  solutions to linear deterministic Volterra equations: 
\begin{lemma}\label{L:Riccatiinvariance}
    Let $\Gamma$ satisfying Assumption~\ref{ass:conditions_kernel_limweak}. 
    Let $v\in\R^d$, $F\in C([0,T],\R^d)$ and $G\in C([0,T],\mathcal M_d(\R))$ be such that $v_i\ge0$, $F_i\ge0$, and $G_{ij}\ge0$ for all $i,j=1,\ldots,d$ with $i\ne j$. Then, the linear Volterra equation 
\begin{equation}\label{E:chi}
\chi(t) =  v+  \int_0^t \Gamma(t,s) \left( F(s) +  G(s)\chi(s) \right) \ud s
\end{equation}
has a unique solution $\chi\in C([0,T],\R^d)$ with $\chi_i\ge0$ for $i=1,\ldots,d$.
\end{lemma}

\begin{proof} We first observe that \eqref{E:chi} can be re-written as 
\begin{align}
    \chi(t) = v + \int_0^t \Gamma(t,s)b(s,\chi(s))  \ud s,
\end{align}
with 
$$ b(t,x) = F(t) + G(t)x, \quad (t,x)\in [0,T]\times \R^d, $$
linear in $x$ with bounded coefficients $F,G$ on $[0,T]$ (by continuity). In addition, for any $t\in [0,T]$  and  $x\in \R^d_+$, for any $i\in \{1,\ldots, d\}$, $x_i=0$ implies that  $b_i(t,x) = F_i(t) + \sum_{j\neq i} G_{ij} x_j \geq 0  $, which is the analogue of the invariance conditions  \eqref{eq:bsigmaRd+} for the set $\mathbb R^d_+$ for time-dependent coefficients $b$ and vanishing diffusion coefficient $\sigma$. Hence, the existence of an $\R^d_+$-valued continuous solution  $\chi$ to \eqref{E:chi} is obtained using a straightforward adaptation  of the proof of Theorem~\ref{thm_maininvariance}  with time-dependent coefficients and $\sigma\equiv 0$, see Remark~\ref{Rk_time_inhomogeneous}.     The uniqueness readily follows from the linearity of the coefficient $b$ in the $x$ variable. 
\end{proof}

We are now in place to derive the existence of a solution to the Riccati--Volterra equation \eqref{RicVolSqrt}.

\begin{lemma}\label{L:Riccatiexistence}
   Assume that  the kernel $\tilde \Gamma $ defined in \eqref{eq:tildeGamma} satisfies Assumption \ref{ass:conditions_kernel_limweak}. For any $f \in C([0,T], \mathbb{R}^d_-)$ the Riccati--Volterra equation \eqref{RicVolSqrt} has a global solution $\psi \in C([0,T], \mathbb{R}^d_-)$, i.e.~ $\psi_i \le 0$, $i = 1, \ldots, d$.
\end{lemma}

\begin{proof} We will establish the existence of a solution $\tilde \psi \in C([0,T], \R^d_{-})$ for the Volterra--Riccati equation written in forward form in \eqref{eq:RiccatiFwd} with the kernel $\tilde \Gamma$. Then,  by a change of variables, the function $\psi  \in C([0,T], \R^d_{-})$ defined by  $\psi(t)=\tilde \psi(T-t)$ solves the Riccati--Volterra equation \eqref{RicVolSqrt}. 

By Lemma~\ref{L:Riccatilocalexistence}, since the kernel $\tilde \Gamma$ satisfies    Assumption~\ref{ass:conditions1_kernel},  there exists a unique non-continuable solution $(\tilde \psi,T_{\rm max})$ of \eqref{eq:RiccatiFwd}. Our aim is to argue that $T_{\rm max} \geq T$ by showing that 
\begin{align}\label{eq:temptmax}
   \sup_{t< T_{\rm max}} |\tilde \psi(t)| < \infty.  
\end{align}
For this, we first observe that  
on the interval $[0,T_{\rm max})$, the function   $-\tilde \psi_i$ satisfies the linear equation
\[
\chi_i(t) = \int_0^t \tilde \Gamma(t,s) \left( -f_i(T-s) + (B^\top \chi(s))_i  + \frac{\sigma_i^2}{2}  \tilde \psi_i(s) \chi_i(s)  \right) \ud s.
\]
Due to \eqref{sqrt2} and since $f$ has nonpositive components and $\tilde \Gamma$ satisfies Assumption~\ref{ass:conditions_kernel_limweak}, Lemma~\ref{L:Riccatiinvariance} yields $\tilde \psi_i\le0$, $i=1,\ldots,d$. Next, let $\ell \in C([0,T],\R^d)$ be the unique solution of the linear equation
\begin{align*}
\ell(t) =  \int_0^t \tilde \Gamma(t,s)\left( f(T-s) + B^\top \ell(s) \right)  \ud s.
\end{align*} Observing that the function  $\tilde \psi - \ell$  satisfies the equation
\[
\chi_i(t) = \int_0^t \tilde \Gamma(t,s)  \left( (B^\top \chi(s))_i + \frac{\sigma_i^2}{2} \tilde \psi^2_i(s)^2  \right)\ud  s,
\]
on $[0,T_{\rm max})$, another  application of Lemma~\ref{L:Riccatiinvariance}   yields that  $\ell_i\le \tilde \psi_i$ on $[0,T_{\rm max})$. In summary, we have shown that
\[
\text{$\ell_i \le \tilde \psi_i \le 0$  on $[0,T_{\rm max})$ for $i=1,\ldots,d$.}
\]
Since $\ell$  is a global solution and thus have finite norm on any bounded interval, this implies \eqref{eq:temptmax} so that $T_{\rm max}\geq T$ as needed. This ends the proof. \end{proof}

\subsection{Putting everything together}

We are now ready to complete the proof of Theorem~\ref{T:VolSqrt}.

\begin{proof}[Proof of  Theorem~\ref{T:VolSqrt}]
The existence of an $\R^d_+$-valued continuous solution $X$ to the stochastic Volterra equation \eqref{eq:VolSqrt} for any $X_0 \in \R^d_+$ has been obtained in Section~\ref{S:existenceRd+}, together with the Hölder regularity of the sample paths of $X$.
Weak uniqueness of $X$ is a consequence of  the exponential-affine transform formula in \eqref{eq:laplace}, as $f$ ranges through $\mathcal C([0,T], \R_-^d)$. It remains to argue \eqref{eq:laplace}. For $f \in \mathcal C([0,T], \R_-^d)$, an application of Lemma~\ref{L:Riccatiexistence} yields the existence of a solution $\psi \in C([0,T], \R^d_-)$ to the Riccati-Volterra equation \eqref{RicVolSqrt}. 
Then, an application of Lemma~\ref{L:verification} gives the exponential-affine transform formula in \eqref{eq:laplace} and ends the proof. 
\end{proof}

\section{Properties and characterization of nonnegativity preserving double kernels}\label{Sec_kernels}

\subsection{Monotonicity condition for nonnegative preserving kernels to deal with $x_0 \geq 0$.}
\begin{prop}\label{prop_pos_gen}
	Let $T\in (0,+\infty]$ and $\Gamma:\Delta_T \to \R_+$ be a double kernel preserving nonnegativity on $[0,T]$ such that $\Gamma(s,s)>0$ for all $s\in \R_+$ with $s\le T$.  We assume moreover that 
	$$ t \mapsto \Gamma(t,s) \text{ is nonincreasing for } t\in[s,T], t<\infty. $$
	Let $K\in \N^*$, $0\le t_1 <\dots <t_K \le T$ and $x_0,\dots,x_K \in \R$ be such that $x_0\ge 0$ and
	$$\forall k \in \{1,\dots, K\}, \ x_0+\sum_{k'=1}^kx_{k'}\Gamma(t_k,t_{k'})\ge 0.$$
	Then, we have $x_0+\sum_{k: t_k\le t} x_k \Gamma(t,t_k)\ge 0$ for all $t\in  [0,T]$, $t<\infty$.
\end{prop}
\begin{proof}
	We introduce $\tilde{x}_1=-x_0/\Gamma(t_1,t_1)$ and $\tilde{x}_k=\frac{-1}{\Gamma(t_k,t_k)}\left( x_0+ \sum_{k'=1}^{k-1} \tilde{x}_{k'}\Gamma(t_k,t_{k'}) \right)$, so that $ x_0+\sum_{k'=1}^k \tilde{x}_{k'}\Gamma(t_k,t_{k'}) = 0$ for all $k\in \{1,\dots, K\}$.
	
	Let $\delta_k=x_k-\tilde{x}_k$ for $k\in \{1,\dots,K\}$. We have  $\forall k \in \{1,\dots, K\}, \ \sum_{k'=1}^k\delta_{k'}\Gamma(t_k,t_{k'})\ge 0$, which gives $\sum_{k: t_k\le t} \delta_k \Gamma(t,t_k)\ge 0$ because $\Gamma$ preserves nonnegativity on $[0,T]$. Since $x_0+\sum_{k: t_k\le t} x_k\Gamma (t_k,t)=x_0+\sum_{k: t_k\le t} \tilde{x}_k\Gamma(t,t_k) + \sum_{k: t_k\le t} \delta_k\Gamma(t,t_k)$, it is then sufficient to check that $x_0+\sum_{k: t_k\le t} \tilde{x}_k\Gamma(t,t_k)\ge 0$ for all $t\ge 0$. To get this, we check easily by induction on~$k$ that $\tilde{x}_k\le 0$ for $k=1,\dots,K$ and use that the functions $\Gamma(\cdot,t_k)$ are nonincreasing.
\end{proof}

\subsection{Characterization}

As in \cite[Section 2.1]{Alfonsi23}, we give a characterization of nonnegativity preserving double kernels. We need few notation to state it. Let $T\in (0,+\infty]$ and we denote by abuse of notation $\Delta_T=\Delta$ for $T=+\infty$. 
For $l\ge 2$, and $T\in \R_+^*$, we define 
\begin{align*}
	&\Delta^l=\{(s_l,\dots,s_1)\in \R_+^l: s_1\le \dots \le s_l  \}, \ \mathring{\Delta}^l=\{(s_l,\dots,s_1)\in \R_+^l: 0\,<\,s_1< \dots < s_l  \}, \\
	&\Delta^l_T=\{(s_l,\dots,s_1)\in \R_+^l: s_1\le \dots \le s_l \le T \}, \ \mathring{\Delta}^l_T=\{(s_l,\dots,s_1)\in \R_+^l: 0\,<\,s_1< \dots < s_l\,<\,T \}.
\end{align*}
Note that for $l=2$, we have  $\Delta_T=\Delta^2_T$ and $\mathring{\Delta}_T=\mathring{\Delta}^2_T$. We make the same abuse of notation  and set $\Delta^l_T=\Delta^l$, $\mathring{\Delta}^l_T=\mathring{\Delta}^l$ for $T=+\infty$.
\begin{definition} 
	Let $T\in (0,+\infty]$ and kernel $\Gamma:\Delta_T \to \R_+$ such that $\Gamma(s,s)>0$ for all $s\ge 0$.  We define by induction, for $l\ge 2$, the functions $\Gamma_l: \Delta^l_T \to \R$ by   $\Gamma_2(s_2,s_1)=\Gamma(s_2,s_1)/\Gamma(s_1,s_1)$ for $(s_2,s_1)\in \Delta_T$ and
	\begin{equation}\label{def_Gammal}
	\Gamma_{l+1}(s_{l+1},\dots,s_{1})=\Gamma_{l}(s_{l+1},\dots,s_3,s_1)-\Gamma_2(s_2,s_1)\Gamma_{l}(s_{l+1},\dots,s_2),
	\end{equation}
	for $l\ge 2$, $(s_{l+1},\dots,s_{1}) \in \Delta_T^{l+1}$.
\end{definition}

\begin{theorem}\label{thm_char_pos}
	Let $T\in (0,+\infty]$ and $\Gamma:\Delta_T \to \R_+$ such that $\Gamma(s,s)> 0$ for $s\ge 0$. The double kernel~$\Gamma$ preserves nonnegativity on $[0,T]$ if, and only if all the functions $\Gamma_l:\Delta^l_T \to \R$, $l\ge 2$, defined by~\eqref{def_Gammal} are nonnegative on $\mathring{\Delta}^l_T$.
\end{theorem}  
\noindent This theorem is a key result to check whether a double kernel preserves nonnegativity. Its proof is postponed to the next subsection. Before that, we state interesting corollaries.

\begin{corollary}\label{cor_expoinvariance}
	Let~$\Gamma:\Delta \to \R_+$ such that $\Gamma(s,s)> 0$ for $s\ge 0$. Let $\rho$ be a Borel measure on $\R_+$ finite on compact sets.  Then, $\Gamma:\Delta \to \R_+$ preserves nonnegativity if, and only if $\Gamma^{\rho}(t,s)=\Gamma(t,s)e^{-\rho((s,t])}$ preserves nonnegativity.
\end{corollary}
\begin{proof}
	We prove the first part and consider the associated functions $\Gamma_l:\Delta^l\to \R$ defined inductively by $\Gamma_2^\rho(s_2,s_1)=\frac{\Gamma^\rho(s_2,s_1)}{\Gamma^\rho(s_1,s_1)}=\Gamma_2(s_2,s_1)e^{-\rho((s_1,s_2])}$ and $$\Gamma^\rho_{l+1}(s_{l+1},\dots,s_1)=\Gamma^\rho_{l}(s_{l+1},\dots,s_3,s_1)-\Gamma^\rho_2(s_2,s_1)\Gamma^\rho_{l}(s_{l+1},\dots,s_2).$$ We get that $\Gamma^\rho_{l+1}(s_1,\dots,s_{l+1})=e^{-\rho((s_1,s_{l+1}])}\Gamma_{l+1}(s_1,\dots,s_{l+1})$: this is true for $l=1$ and then obvious by induction. Therefore $\Gamma_l\ge 0 \iff \Gamma^\rho_l\ge 0$, and we conclude by Theorem~\ref{thm_char_pos}. 
\end{proof}

\begin{remark}
    Let $\Gamma:\Delta \to \R_+$ be a kernel that preserves nonnegativity and  satisfies Assumption~\ref{ass:conditions1_kernel}. Let $f:\R_+^* \to \R_+$ be such that $\R_+ \ni t\mapsto \int_0^t f(u) {\ud u}$ is locally H\"older continuous.  Then, the kernel $\tilde{\Gamma}(t,s)=\Gamma(t,s)e^{-\int_s^t f(u){\ud u}}$ preserves nonnegativity by Corollary~\ref{cor_expoinvariance} and besides satisfies Assumption~\ref{ass:conditions1_kernel}. Indeed, for $T>0$ and $(t,s)\in \Delta_T$, we have
    \begin{align*}
&\int_s^t {\tilde{\Gamma}(t,u)^2\,\ud u} + \int_0^s {(\tilde{\Gamma}(t,u) - \tilde{\Gamma}(s,u))^2\,\ud u} \\
&\le \int_s^t {\Gamma(t,u)^2\,\ud u} + 2 \int_0^s e^{-2\int_u^s f(v){\ud v}} {(\Gamma(t,u) - \Gamma(s,u))^2\,\ud u}        +2\int_0^s \left(1-e^{-\int_s^t f(v){\ud v}} \right)^2{\Gamma(t,u)^2\,\ud u}     \\&
\le \int_s^t {\Gamma(t,u)^2\,\ud u} + 2 \int_0^s  {(\Gamma(t,u) - \Gamma(s,u))^2\,\ud u}        + 2 \int_0^s \left(\int_s^t f(v){\ud v} \right)^2{\Gamma(t,u)^2\,\ud u}.
    \end{align*}
    Since $\Gamma$ satisfies Assumption~\ref{ass:conditions1_kernel},  there exist $\eta>0,\gamma\in(0,1/2]$ such that the two first terms are upper bounded by $2\eta(t-s)^{2\gamma}$. Possibly considering a smaller $\gamma>0$, we may assume that there exists $C\in \R_+$ such that $\left|\int_s^t f(v){\ud v} \right|\le C(t-s)^\gamma$ for $(t,s)\in \Delta_T$. The third term is then upper bounded by $2\eta T^{2\gamma} C^2(t-s)^{2\gamma}$ since $\int_0^s {\Gamma(t,u)^2\,\ud u}\le \int_0^t {\Gamma(t,u)^2\,\ud u}\le \eta T^{2\gamma}$.
    \end{remark}

\begin{corollary}\label{C:tildegammapreserving}
	Let $T\in \R_+^*$ and $\Gamma:\Delta_T \to \R_+$ be a double kernel that preserves nonnegativity on $[0,T]$ and such that $\Gamma(s,s)=\gamma >0$ is constant for $s\in [0,T]$. Then, the double kernel $\tilde{\Gamma}(t,s)=\Gamma(T-s,T-t)$ for $(t,s)\in \Delta_T$ preserves nonnegativity on $[0,T]$. 
\end{corollary}
\begin{proof}
	We prove that $\tilde{\Gamma}_l(s_l,\dots,s_1)=\Gamma_l(T-s_1,\dots,T-s_l)$, which shows then the claim by Theorem~\ref{thm_char_pos}.  
	
	For $l=2$, we have for $(s_2,s_1) \in \mathring{\Delta}^2_T$, $$\tilde{\Gamma}_2(s_2,s_1)=\frac{\Gamma(T-s_1,T-s_2)}{\Gamma(T-s_1,T-s_1)}=\frac{\Gamma(T-s_1,T-s_2)}{\Gamma(T-s_2,T-s_2)}=\Gamma_2(T-s_1,T-s_2),$$
	by using that $\Gamma(T-s_1,T-s_1)=\Gamma(T-s_2,T-s_2)=\gamma$.
	
	For $l> 2$, we prove by induction the following formula 
	\begin{align}\label{Gamma_l_nonrec}
	\Gamma_{l}(s_l,\dots,s_1)=\Gamma_2(s_l,s_1)+\sum_{j=1}^{l-2} (-1)^j \sum_{1<l'_1\dots<l'_j<l}\Gamma_2(s_l,s_{l'_j})\times\dots\times \Gamma_2(s_{l'_1},s_1), 
	\end{align}
	which then directly gives $\tilde{\Gamma}_l(s_l,\dots,s_1)=\Gamma_l(T-s_1,\dots,T-s_l)$ by using that $\tilde{\Gamma}_2(t,s)=\Gamma_2(T-s,T-t)$ for $0<s<t<T$. The formula is clearly true for $l=3$. Let us assume it true for $l$, then we have by~\eqref{def_Gammal}
	\begin{align}
	\Gamma_{l+1}(s_{l+1},s_l,\dots,s_1)=&\Gamma_2(s_{l+1},s_1)+\sum_{j=1}^{l-2} (-1)^j \sum_{2<l'_1\dots<l'_j<l+1}\Gamma_2(s_{l+1},s_{l'_j})\dots \Gamma_2(s_{l'_1},s_1) \\& -\Gamma_2(s_2,s_1)\left(\Gamma_2(s_{l+1},s_2)+\sum_{j=1}^{l-2} (-1)^j \sum_{2<l'_1\dots<l'_j<l+1} \Gamma_2(s_{l+1},s_{l'_j})\dots \Gamma_2(s_{l'_1},s_2)\right).
	\end{align}
	We note that  
	$$ - \sum_{2<l'_1<l+1}\Gamma_2(s_{l+1},s_{l'_1})\Gamma_2(s_{l'_1},s_1) -\Gamma_2(s_2,s_1) \Gamma_2(s_{l+1},s_2)= - \sum_{1<l'_1<l+1}\Gamma_2(s_{l+1},s_{l'_1})\Gamma_2(s_{l'_1},s_1),$$
	and for $j>1$,
	\begin{align*}
	& (-1)^j \sum_{2<l'_1\dots<l'_j<l+1}\Gamma_2(s_{l+1},s_{l'_j})\dots \Gamma_2(s_{l'_1},s_1) - \Gamma_2(s_2,s_1) (-1)^{j-1}\sum_{2<l'_1\dots<l'_{j-1}<l+1}\Gamma_2(s_{l+1},s_{l'_{j-1}})\dots \Gamma_2(s_{l'_1},s_1)\\
	&=(-1)^j \sum_{1<l'_1\dots<l'_j<l+1}\Gamma_2(s_{l+1},s_{l'_j})\dots \Gamma_2(s_{l'_1},s_1),
	\end{align*}
	which proves~\eqref{Gamma_l_nonrec}.
\end{proof}

We finally prove in this subsection Corollary~\ref{cor_convolution} that allows to get nonnegativity preserving double kernels from any nonnegativity preserving convolution kernel.
\begin{proof}[Proof of Corollary~\ref{cor_convolution}]
	Let us define $G_1(a)=G(a)/G(0)$ for $a>0$ and, for $l\ge 2$ and $a_1,\dots,a_l>0$, $G_{l}(a_1,\dots,a_l)=G_{l-1}(a_1,\dots,a_{l-2},a_{l-1}+a_{l})-G_1(a_l)G_{l-1}(a_1,\dots,a_{l-2},a_{l-1})$.  From~\cite[Theorem 2.6]{Alfonsi23}, these functions are nonnegative since $G$ preserves nonnegativity. We now extend these functions to $a_1,\dots,a_l\ge 0$ with the same induction formula. We still have $G_l(a_1,\dots,a_l)\ge 0$: indeed, we check easily by induction on~$l$ that $G_l(a_1,\dots,a_l)=0$ if $a_1=0$ or $a_{l}=0$ and $G_l(a_1,\dots,a_l)=G_{m}(a_1, a_{n_2},\dots,a_{n_{m-1}}, a_l)$ where $1=n_1<n_2\dots<n_m=l$ are the indices for which $a$ is positive, i.e. $\{n_k,1\le k \le m\}=\{i : a_i>0\}$.   
	
	We now calculate the functions $\Gamma^r_l$. We have $\Gamma_2^r(s_2,s_1)=G(\rho((s_1,s_2]))/G(0)=G_1(\rho((s_1,s_2]))$ for $0<s_1<s_2$. We now prove by induction on $l$ that for $l\ge 2$ and $0<s_1<\dots<s_l$, 
	$$ \Gamma^r_l(s_l,\dots,s_1)=G_{l-1}\left(\rho((s_{l-1},s_l]), \dots,\rho((s_1,s_2])\right).$$
	Let $l\ge  2$ and $0<s_1<\dots<s_{l+1}$. We have from~\eqref{def_Gammal}
	\begin{align*}
		&\Gamma_{l+1}^r(s_{l+1},\dots, s_1)= \Gamma_{l}^r(s_{l+1},\dots,s_3, s_1)-\Gamma_2^r(s_2,s_1)\Gamma_{l}^r(s_{l+1},\dots, s_2) \\
		&=G_{l-1}\left(\rho((s_{l},s_{l+1}]), \dots,\rho((s_3,s_4]),\rho((s_1,s_3])\right)-G_1(\rho((s_1,s_2]))G_{l-1}\left(\rho((s_{l},s_{l+1}]), \dots,\rho((s_2,s_3])\right) \\
		&=G_{l}\left(\rho((s_{l},s_{l+1}]), \dots,\rho((s_1,s_2])\right),
	\end{align*}
	by using the induction hypothesis and then the definition of $G_{l}$ with $\rho((s_1,s_3])=\rho((s_1,s_2])+\rho((s_2,s_3])$. We show similarly that $ \Gamma^\ell_l(s_l,\dots,s_1)=G_{l-1}\left(\rho([s_{l-1},s_l)), \dots,\rho([s_1,s_2))\right)$, which shows that $\Gamma^\ell_l,\Gamma^r_l\ge 0$ for any $l\ge 2$. This gives the claim by Theorem~\ref{thm_char_pos}.
\end{proof}

\subsection{Proof of the characterization (Theorem~\ref{thm_char_pos})}
For a double kernel~$\Gamma:\Delta_T \to \R_+$, we denote
\begin{align}
\EKG(T):=\Bigg\{ (x_1,\dots,x_K,t_1,\dots,t_K) \in \R^{2K}& : 0\le t_1< \dots<t_K<T  \text{ and } \label{def_EKG}\\ 
&\forall k \in \{1, \dots, K\},\ \sum_{k'=1}^k x_{k'}\Gamma(t_k,t_{k'})\ge 0 \Bigg\}.\notag
\end{align}

\begin{lemma}\label{lem_char_pos} Let $T\in (0,+\infty]$ and $\Gamma:\Delta_T \to \R_+$ such that $\Gamma(s,s)> 0$ for $s\ge 0$. Let $K\in \N^*$ and  $0\le t_1 <\dots <t_K < t_{K+1}=t \le T$ with $t<+\infty$. Then,
	\begin{align}& \left( (x_1,\dots,x_K,t_1,\dots,t_K)\in \EKG(T) \implies \sum_{k=1}^K x_k\Gamma(t,t_k)\ge 0 \right) \label{implic} \\ \iff &\forall i\in \{1,\dots,K\}, \Gamma_{i+1}(t_{K+1},\dots,t_{K+1-i})\ge 0. \notag
	\end{align}
	Let $0\le t_1 <\dots <t_K<T$ and define $x_1=1$ and  $x_k=\frac{-1}{\Gamma(t_k,t_k)} \sum_{k'=1}^{k-1} x_{k'}\Gamma(t_k,t_{k'})$ for $k\ge 2$, so that $\sum_{k'=1}^k x_{k'}\Gamma(t_k,t_{k'})=0$. Then, for $k\in \{1,\dots,K-1\}$, we have
	$$x_{k+1}=-\frac{\Gamma(t_1,t_1)}{\Gamma(t_{k+1},t_{k+1})}\Gamma_{k+1}(t_{k+1},\dots,t_1).$$
\end{lemma}
Lemma~\ref{lem_char_pos} allows to characterize nonnegativity preserving double kernels, which amounts to having the implication~\eqref{implic} for all $K\in \N$ and $0\le t_1<\dots< t_K< t$. 

\begin{proof}[Proof of Lemma~\ref{lem_char_pos}.]
	We consider the following minimization problem,  for $t> t_K$:
	\begin{equation}  \label{min_pb} \inf_{(x_1,\dots,x_K): (x_1,\dots,x_K,t_1,\dots,t_K)\in \EKG}\ \sum_{k=1}^Kx_k\Gamma(t,t_k).
	\end{equation}
	Since the set of constraints is linear and triangular, we can write the linear function to optimize as a linear combination of the constraints:
	\begin{equation}\label{lin_to_opt}
	\sum_{k=1}^Kx_k\Gamma(t,t_k)= \sum_{k=1}^K \beta_k \left( \sum_{k'=1}^k x_{k'}\Gamma (t_k,t_{k'})\right), 
	\end{equation}
	with
	\begin{equation}\label{def_betas}
	\forall k' \in \{1,\dots, K \}, \sum_{k=k'}^K \beta_k \Gamma(t_k,t_{k'})  =\Gamma(t,t_{k'}).
	\end{equation}
	This leads to
	\begin{equation}
	\begin{cases}
	\sum_{k=1}^K \beta_k \Gamma_2(t_k,t_1)  =\Gamma_2(t,t_1)\\
	\sum_{k=2}^K \beta_k \Gamma_2(t_k,t_2)  =\Gamma_2(t,t_2) \\
	\phantom{\sum_{k=2}^K \beta_k G_1(t_k-t_2) }  \vdots \\
	\beta_{K-1}+\beta_K \Gamma_2(t_K,t_{K-1})=\Gamma_2(t,t_{K-1})\\
	\beta_K=\Gamma_2(t,t_K).
	\end{cases}
	\end{equation}
	Using the Gauss elimination method, we get by replacing $\beta_K$ by $\Gamma_2(t,t_K)$:
	\begin{equation}
	\begin{cases}
	\sum_{k=1}^{K-1} \beta_k \Gamma_2(t_1,t_k)  =\Gamma_3(t,t_K,t_1)\\
	\sum_{k=2}^{K-1} \beta_k \Gamma_2(t_k,t_2)  =\Gamma_3(t,t_K,t_2) \\
	\phantom{\sum_{k=2}^K \beta_k G_1(t_k-t_2) }  \vdots \\
	\beta_{K-2}+ \beta_{K-1}\Gamma(t_{K-1},t_{K-2})=\Gamma_3(t,t_K,t_{K-2})\\
	\beta_{K-1}=\Gamma_2(t,t_{K-1})-\Gamma_2(t,t_K) \Gamma_2(t_K,t_{K-1})=\Gamma_3(t,t_K,t_{K-1})\\
	\beta_K=\Gamma_2(t,t_K).
	\end{cases}
	\end{equation}
	Going on the Gauss elimination, we end up with
	$$\beta_{K+1-i}=\Gamma_{i+1}(t,t_K,\dots,t_{K+1-i}), \ i \in \{1,\dots, K\}.$$ 
	To conclude, it remains to observe that the infimum~\eqref{min_pb} is nonnegative if, and only if $\beta_1,\dots,\beta_K\ge 0$.
	
	We now prove the second part of the statement. Let us consider  $x_1=1$, and $x_k=-\frac{1}{\Gamma(t_k,t_k)}\sum_{k'=1}^{k-1}\Gamma(t_k,t_{k'})x_{k'}$ for $k\ge 2$, so that $\sum_{k'=1}^k\Gamma(t_k,t_{k'})x_{k'}=0$ for $k\ge 2$. From~\eqref{lin_to_opt}, we get for $t=t_{K+1}$ that $x_{K+1}$ satisfies:
	\begin{align*}
	\Gamma(t_{K+1},t_{K+1})x_{K+1}=-\sum_{k=1}^{K}\Gamma(t_{K+1},t_K)x_{k}&=-\sum_{k=1}^K \beta_k \left( \sum_{k'=1}^k x_{k'}\Gamma(t_k,t_{k'})\right) \\& = - \beta_1\Gamma(t_1,t_1)= -\Gamma(t_1,t_1)\Gamma_{K+1}(t_{K+1},\dots,t_1),
	\end{align*}
	leading to $x_{K+1}=-\frac{\Gamma(t_1,t_1)}{\Gamma(t_{K+1},t_{K+1})}\Gamma_{K+1}(t_{K+1},\dots,t_1)$. The parameter $K\in \N^*$ being arbitrary, this gives the claim.
\end{proof}
\begin{proof}[Proof of Theorem~\ref{thm_char_pos}] The double kernel $\Gamma$ preserves nonnegativity if the infimum~\eqref{min_pb} is nonnegative for any $K\in \N^*$, any $0 \le t_1<\dots<t_K<t\le T$ with $t<+\infty$, which gives the necessary and sufficient condition by Lemma~\ref{lem_char_pos}.
\end{proof}

\subsection{Completely monotone double kernels}
We prove in this paragraph that completely monotone double kernels preserve nonnegativity. 
\begin{proof}[Proof of Theorem~\ref{thm:completely_monotone_double}] The proof is written in the first case $\Gamma(t,s) = \int_{\R} {e^{-\rho(\alpha, (s,t])}\,\mu(\ud\alpha)}$, and is analogous for the second one. It generalizes the proof of~\cite[Theorem 2.11]{Alfonsi23} to double kernels. 
	By Theorem~\ref{thm_char_pos}, it is sufficient to check that the functions~$\Gamma_l$ are nonnegative. To do so, we use the second statement of Lemma~\ref{lem_char_pos}. Let $(t_k)_{k\in \N^*}$ be an increasing sequence of nonnegative real numbers, $x_1=1$ and  $x_k=\frac{-1}{\Gamma(t_k,t_k)} \sum_{k'=1}^{k-1} x_{k'}\Gamma(t_k,t_{k'})$ for $k\ge 2$. Our goal is to prove that $x_k\le 0$ for all~$k\ge 2$, which gives by Lemma~\ref{lem_char_pos} the nonnegativity of the functions $\Gamma_l$ since the sequence $(t_k)_{k\in \N^*}$ is arbitrary. 
	
	We define, for $\alpha\in \R$, $X^\alpha_t:=\sum_{k: t_k\le t}x_k\,e^{-\rho(\alpha, (t_k, t])}$ and $X_t:=\int_{\R} X^\alpha_t\,\mu(\ud\alpha)$, so that $X_t=\sum_{k: t_k\le t}x_k\,\Gamma(t,t_k)$. In particular, we have $X^\alpha_{t_k}=\sum_{j=1}^k x_j\,e^{-\rho(\alpha, (t_j, t_k])}$ for every $k\geq1$ and $X_{t_k}=0$ for every $k\geq2$. Let $A_k := \{\alpha\in\R:X_{t_k}^{\alpha}>0\}$, we show by induction on $k\geq2$ that
	\begin{itemize}
		\item $x_k\leq0$;
		\item $A_k \subset A_{k-1}$ and either of the following two conditions holds:
		\begin{itemize}
			\item if $\mu(A_k)=0$, then $X_{t_k}^{\alpha} = 0$ $\mu$-a.e.;
			\item if else $\mu(A_k)>0$, then there exists $a_k\in\R$ such that $A_k = (-\infty, a_k)$ or $(-\infty, a_k]$, and $A_k\ni\alpha \mapsto X_{t_k}^{\alpha}$ is nonincreasing.
		\end{itemize}
	\end{itemize}
	
	We first treat the case $k=2$. By construction $x_2 = -\frac{1}{\mu(\R)}\int_{\R} {e^{-\rho(\alpha, (t_1, t_2])}\,\mu(\ud\alpha)}<0$ and, since $X_{t_1}^{\alpha} = 1$ for all $\alpha\in\R$, we have $A_2 \subset A_1 = \R$. If $\mu(A_2)=0$, then $X_{t_2}^{\alpha}\leq0$ $\mu$-a.e. and, since $X_{t_2} = \int_{\R} {X_{t_2}^{\alpha}\,\mu(\ud\alpha)} = 0$, we have $X_{t_2}^{\alpha}=0$ $\mu$-a.e.. Suppose now that $\mu(A_2)>0$. Observing that 
	\begin{equation*}
	X_{t_2}^{\alpha} = e^{-\rho(\alpha, (t_1, t_2])} + x_2,
	\end{equation*}
	for all $\alpha\in\R$, the monotonicity of $\alpha\mapsto\rho(\alpha,\cdot)$ ensures that $\alpha\mapsto X_{t_2}^{\alpha}$ is nonincreasing on $A_1=\R$, in particular on $A_2$. Then, there exists $a_2\in\R$ such that $A_2 = (-\infty, a_2)$ or $(-\infty, a_2]$.
	
	Let us at present assume that the induction hypothesis is valid for some $k>2$. We first prove that $x_{k+1} \leq 0$. Observe that by construction,
	\begin{equation}\label{eq:xk+1}
	x_{k+1} = -\frac{1}{\mu(\R)}\int_{\R} {e^{-\rho(\alpha, (t_k, t_{k+1}])}\,X_{t_k}^{\alpha}\,\mu(\ud\alpha)}.
	\end{equation}
	Using the induction hypothesis, we get from \eqref{eq:xk+1} that $x_{k+1}=0$ if $\mu(A_k)=0$. If else $\mu(A_k)>0$, since $A_k = (-\infty, a_k)$ or $(-\infty, a_k]$ by hypothesis, we have $X_{t_k}^{\alpha} > 0$ for $\alpha < a_k$ and $X_{t_k}^{\alpha} \leq 0$ for $\alpha > a_k$. Using the monotonicity of $\alpha\mapsto\rho(\alpha,\cdot)$, it holds that $e^{-\rho(\alpha, (t_k, t_{k+1}])}\,X_{t_k}^{\alpha} \geq e^{-\rho(a_k, (t_k, t_{k+1}])}\,X_{t_k}^{\alpha}$ for all $\alpha\in\R$. Injecting this into \eqref{eq:xk+1},
	\begin{equation*}
	-x_{k+1} \geq \frac{1}{\mu(\R)}e^{-\rho(a_k, (t_k, t_{k+1}])}\int_{\R} {X_{t_k}^{\alpha}\,\mu(\ud\alpha)} =0,
	\end{equation*}
	where $X_{t_k} = \int_{\R} {X_{t_k}^{\alpha}\,\mu(\ud\alpha)} = 0$. Consider then $A_{k+1} := \{\alpha\in\R:X_{t_{k+1}}^{\alpha}>0\}$. Observing that 
	\begin{equation*}
	X_{t_{k+1}}^{\alpha} = x_{k+1} + e^{-\rho(\alpha, (t_k, t_{k+1}])}\,X_{t_k}^{\alpha},
	\end{equation*}
	we get that if $\alpha \notin A_k$, then $X_{t_{k+1}}^{\alpha} \leq 0$ as $x_{k+1} \leq 0$ and $X_{t_k}^{\alpha} \leq 0$, hence $\alpha \notin A_{k+1}$. This ensures that $A_{k+1} \subset A_k$. As before, if $\mu(A_{k+1})=0$, then $X_{t_{k+1}}^{\alpha}\leq0$ $\mu$-a.e. and then $X_{t_{k+1}}^{\alpha}=0$ $\mu$-a.e. as $X_{t_{k+1}}=0$. If else $\mu(A_{k+1}) > 0$, then we must have $\mu(A_k) > 0$ as $A_{k+1} \subset A_k$. By the induction hypothesis, it holds that $\alpha \mapsto X_{t_k}^{\alpha}$ is nonincreasing on $A_k$ with $A_k = (-\infty, a_k)$ or $(-\infty, a_k]$. By using again the monotonicity of $\alpha\mapsto\rho(\alpha,\cdot)$, we have that $\alpha \mapsto X_{t_{k+1}}^{\alpha}$ is nonincreasing on $A_k$ as well, then on $A_{k+1}$. Hence, there exists $a_{k+1}\le a_k$ such that $A_{k+1} = (-\infty, a_{k+1})$ or $(-\infty, a_{k+1}]$, thus proving the induction hypothesis.
\end{proof}

\appendix 

\section{Invariance/viability for SDEs}\label{A:invarianceSDE}
In this appendix, we collect some characterizations of invariance and viability for the stochastic differential equation \eqref{SDE_xi_lambda} for $\lambda >0$, that is conditions on $(b,\sigma)$ that are equivalent to \eqref{WSI_lambda}.  We denote by $C(x)=\sigma(x)\sigma(x)^\top$ and we recall that $\mathscr{C}$ is a closed convex set of $\mathbb R^d$ and we assume Assumption~\ref{ass:conditions1_coefficients}. In words, at the boundary points, the diffusion matrix has to be tangential to boundary and a compensated drift needs to be inward pointing. 

\begin{enumerate}
	\item 
	 \cite{bardi2002geometric} use  Nagumo-type geometric conditions on the second order normal cone: their main result states that the closed set $\mathscr{C}$ is  {stochastically invariant} for \eqref{SDE_xi_lambda} if and only if
 \begin{equation*}\label{eqbardi}
\lambda u^\top b(x) + \frac{\lambda^2}{2} \mbox{Tr}(vC(x)) \leq 0, \quad x \in \mathscr{C} \mbox{ and } (u,v) \in \mathcal{N}^2_{\mathscr{C}}(x), 
 \end{equation*}
and $ \mathcal{N}^2_{\mathscr{C}}(x)$ is the second order normal cone at the point $x$: 
\begin{equation}\label{eq: def second order cone}	
\mathcal{N}^2_{\mathcal{\mathscr{C}}}(x) :=\left\{(u,v) \in \mathbb{R}^d \times \mathbb{S}^d: \langle u, y-x \rangle + \frac{1}{2} \langle y-x , v (y-x) \rangle\leq 0, {\forall\; y \in \mathscr{C}} \right\}.
\end{equation}
Here, $\mathbb{S}^d$ {stands for} the cone of symmetric {$d\times d$} matrices. 
This also corresponds to the positive maximum principle of \cite{ethier2009markov}.
\item
\cite{doss1977liens} and \cite{da2004invariance} {give} necessary and sufficient conditions for  the  stochastic invariance in terms of  the Stratonovich drift - whenever $\sigma$ is differentiable - and the first order normal cone:
\begin{equation}\label{dapratocondintro}
	\sigma(x)^\top u =0   \mbox{ and }  
	\langle u, \lambda b(x)-\frac{\lambda^2}{2} \sum_{j=1}^{d} D\sigma^j(x)\sigma^j(x)  \rangle \leq 0, \quad  x \in \mathscr{C}, \; u \in \mathcal{N}^1_{\mathscr{C}}(x),
	\end{equation}	  
	where $\sigma^j(x)$ denotes the $j$-th column of the matrix $\sigma(x)$, $D\sigma^j$ is the Jacobian of $\sigma^j$, and the first order normal cone   $\mathcal{N}^1_{\mathcal{D}}(x)$ at $x$ is defined as
\begin{equation}\label{eq: def first order cone}
		\mathcal{N}^1_{\mathscr{C}}(x) :=\left\{u \in \mathbb{R}^d: \langle u, y-x \rangle \leq  0, {\forall\; y \in \mathscr{C}} \right\}.
\end{equation}
\item \cite*{abi2019stochastic} provide a first order characterization under weaker regularity assumptions on $\sigma$, assuming that $C=\sigma\sigma^{\top}$ is differentiable: 
\begin{equation}\label{eq: nec suff cond intro}
	C(x) u =0   \mbox{ and }  
	\langle  u, \lambda b(x)-\frac{\lambda^2}{2} \sum_{j=1}^{d} D C^j(x)(CC^+)^j(x)    \rangle \leq 0, \quad x \in \mathscr{C} \mbox{ and } u \in \mathcal{N}^1_{\mathscr{C}}(x).
	\end{equation}
 Here,  $(CC^+)^j(x)$ is the $j$-th column of $(CC^{+})(x)$ with $C(x)^+$ defined as the Moore-Penrose pseudoinverse of $C(x)$. An advantage of such formulation over the Stratonovich one, is that it covers the case of square-root diffusions of the form $\sigma(x) = \sqrt{f(x)}$ with $f$ differentiable, for instance $\sigma(x)= \sqrt{x}$. 
\end{enumerate}

In practice, the first order normal  cone is much  simpler to compute than the second order {cone}. However, the price to pay is to impose a stronger regularity {conditions} on the diffusion or covariance matrices $\sigma$ and $C$.

\section{Kernel approximation lemma}\label{A:proofkernel}

\begin{proof}[Proof of Lemma~\ref{lem_approx_kernel}]

We first define for $s,t\in[0,T]$:
$$\tilde{\Gamma}_M(t,s)=\sum_{i=1}^{2^M} \mathbf{1}_{s \in [(i-1)\frac T {2^M}, i \frac T {2^M}[ } \frac{2^M}{T}\int_{(i-1)\frac T {2^M}}^{i\frac T {2^M}} \Gamma(t,u) \ud u,$$
where we set $\Gamma(t,u)=0$ if $u>t$. Note that it is piecewise constant with respect to~$s$, and that $\tilde{\Gamma}_M(t,s)=0$ for $s\ge i\frac T {2^M}$ when $t\in[ (i-1)\frac T {2^M}, i \frac T {2^M}[$. Let $U$ be a uniform random variable on~$[0,T]$. We have $\tilde{\Gamma}_M(t,U)=\E[\Gamma(t,U)|\lfloor 2^M U/T\rfloor ]$ almost surely since $\E[\Gamma(t,U)|\lfloor 2^M U/T\rfloor =j]=\E[\Gamma(t,U)|U \in [j \frac{T}{2^M},(j+1) \frac{T}{2^M})]=\frac{2^M}{T}\int_{j \frac T {2^M}}^{(j+1)\frac T {2^M}} \Gamma(t,u) \ud u$. Besides, it is a square integrable $(\mathcal{G}_M)$-martingale with $\mathcal{G}_M=\sigma(\lfloor 2^M U/T\rfloor)$  since $\E[\tilde{\Gamma}_M(t,U)^2]\le \E[\Gamma(t,U)^2]$ by Jensen inequality and we have $\mathcal{G}_M\subset \mathcal{G}_{M+1}$ from $\lfloor 2^M U/T\rfloor=j\iff \lfloor 2^{M+1} U/T\rfloor\in\{2j,2j+1\}$. Therefore, we obtain for $t\in[0,T]$ 
\begin{align*}
\int_0^t(\Gamma(t,s)-\tilde{\Gamma}_M(t,s))^2 \ud s &\le \int_0^T(\Gamma(t,s)-\tilde{\Gamma}_M(t,s))^2 \ud s \\
&= T\E \left[\left(\Gamma(t,U) - \E[\Gamma(t,U)|\lfloor 2^M U/T\rfloor ] \right)^2\right]\to_{M\to \infty} 0,
\end{align*}
by Lévy's upward lemma. The function $\tilde{\Gamma}_M$ is thus a piecewise constant approximation of~$\Gamma$ for the norm $L^2$. We now construct  $\hat{\Gamma}_M(t,s)$ that also approximates $\Gamma$ and is besides continuous. 

For $M\in \N$, let $\psi_M(t):\R \to \R_+$ be defined by $\psi_M(t)=0$ if $t \not \in [0,1]$, $\psi_M(t)=1$ if $t  \in [1/(M+2),1-1/(M+2)]$, $\psi_M(t)=(M+2)t$ if $t  \in [0,1/(M+2)]$ and $\psi_M(t)=(M+2)(1-t)$ if $t  \in [1-1/(M+2),1]$. It is continuous, piecewise linear and such that $0\le \psi_M\le \mathbf{1}_{[0,1]}$, $\psi_M(t)\to_{M\to \infty} \mathbf{1}_{[0,1]}(t)$ for any $t\in \R$. We set
$$\hat{\Gamma}_M(t,s)=\sum_{i=1}^{2^M}  \psi_M\left(\frac{2^M}{T}\left(s-(i-1)\frac{T}{2^M}\right) \right) \frac{2^M}{T}\int_{(i-1)\frac T {2^M}}^{i\frac T {2^M}} \Gamma(t,u) \ud u,$$
and we have 
\begin{align*}&\int_0^T(\tilde{\Gamma}_M(t,s)-\hat{\Gamma}_M(t,s))^2 \ud s \\&= \sum_{i=1}^{2^M}   \left(\frac{2^M}{T}\int_{(i-1)\frac T {2^M}}^{i\frac T {2^M}} \Gamma(t,u) \ud u\right)^2 \int_{(i-1)\frac T {2^M}}^{i\frac T {2^M}} \left(1-\psi_M\left(\frac{2^M}{T}\left(s-(i-1)\frac{T}{2^M}\right) \right)\right)^2 \ud s\\
	&= \int_0^1(1-\psi_M(v))^2 \ud v\sum_{i=1}^M   \frac{2^M}{T}\left(\int_{(i-1)\frac T {2^M}}^{i\frac T {2^M}} \Gamma(t,u) \ud u\right)^2\\
	&\le \frac{2}{M+2}\times \int_0^1 \Gamma(t,u)^2 du \to_{M\to \infty} 0,
\end{align*}
by using Cauchy-Schwarz inequality and that $(1-\psi_M)^2$ vanishes on $[1/(M+2),1-1/(M+2)]$ and is upper bounded by~$1$.

We now show that $\hat{\Gamma}_M$ is continuous on $[0,T]^2$. Since $\psi_M$ is continuous, it is sufficient to check for any $i$ that $t \mapsto \int_{(i-1)\frac T {2^M}}^{i\frac T {2^M}} \Gamma(t,u) \ud u$ is continuous. By Assumption~\ref{ass:conditions1_kernel}, $[0,T]\ni t\mapsto \Gamma(t,\cdot) \in L^2([0,T])$ is continuous, which implies that $\int_0^T|\Gamma(t,u)-\Gamma(s,u)| \ud u \to 0$ as $s\to t$, and thus the continuity of $\hat{\Gamma}_M$ on $[0,T]^2$. 

Therefore, $\Gamma_M:\Delta_T\to \R$ defined by $\Gamma_M(t,s)=\hat{\Gamma}_M(t,s)$ for $(t,s)\in\Delta_T$ is also continuous. For $t \in [0,T]$, we have 
$$\int_0^t (\Gamma(t,s)-\Gamma_M(t,s))^2 \ud s \le 2\int_0^T (\Gamma(t,s)-\tilde{\Gamma}_M(t,s))^2 \ud s  +
2\int_0^T (\tilde{\Gamma}_M(t,s)-\hat{\Gamma}(t,s))^2 \ud s \to 0,$$
We have, for $0\le s\le t \le T$, 
\begin{align*}
	&\int_s^t {\Gamma_M(t,u)^2\,\ud u} + \int_0^s {(\Gamma_M(t,u) - \Gamma_M(s,u))^2\,\ud u} \le \int_s^t {\tilde{\Gamma}_M(t,u)^2\,\ud u} + \int_0^s {(\tilde{\Gamma}_M(t,u) - \tilde{\Gamma}_M(s,u))^2\,\ud u},
\end{align*}
by using that $0\le \psi_M\le 1$. We now assume that $(i-1)\frac{T}{2^M}\le s\le t \le  i \frac{T}{2^M}$ for some $i \in \{1,\dots,2^M\}$. Then, we have 
\begin{align*}
	&\int_s^t {\tilde{\Gamma}_M(t,u)^2\,\ud u} + \int_0^s {(\tilde{\Gamma}_M(t,u) - \tilde{\Gamma}_M(s,u))^2\,\ud u}\\
	&=\left(t-s\right)\left(\frac{2^M}{T} \int_{(i-1)\frac T {2^M}}^{i\frac T {2^M}} \Gamma(t,u)  \ud u\right)^2+ \sum_{j=1}^{i-1} \frac{2^M}{T}\left(\int_{(j-1)\frac T {2^M}}^{j\frac T {2^M}} \Gamma(t,u) -\Gamma(s,u) \ud u\right)^2 \\
	&+ \left(s-(i-1)\frac T {2^M}\right)\left(\frac{2^M}{T} \int_{(i-1)\frac T {2^M}}^{i\frac T {2^M}} \Gamma(t,u) -\Gamma(s,u) \ud u\right)^2 \\
	&\le \frac{2^M}{T}\left(t-s\right)\int_{(i-1)\frac T {2^M}}^{t} \Gamma(t,u)^2  \ud u+ \sum_{j=1}^{i-1} \int_{(j-1)\frac T {2^M}}^{j\frac T {2^M}} \left(\Gamma(t,u) -\Gamma(s,u) \right)^2 \ud u\\
	&+ \frac{2^M}{T} \left(s-(i-1)\frac T {2^M}\right) \int_{(i-1)\frac T {2^M}}^{t} \left(\Gamma(t,u) -\Gamma(s,u)\right)^2 \ud u,
\end{align*}
from Cauchy-Schwarz inequality, and using that $\Gamma(t,u)=0$ if $u>t$. We now use Assumption~\ref{ass:conditions1_kernel} for the first term, and also for the two other terms since they are upper bounded by $\int_0^t(\Gamma(t,u)-\Gamma(s,u))^2 \ud u$ to finally get
\begin{align*}
	\int_s^t {\Gamma_M(t,u)^2\,\ud u} + \int_0^s {(\Gamma_M(t,u) - \Gamma_M(s,u))^2\,\ud u} &\le \frac{2^M}{T}\left(t-s\right) \eta (t-(i-1)\frac T {2^M})^{2\gamma}+ \eta |t-s|^{2\gamma}\\&\le 2 \eta |t-s|^{2\gamma},
\end{align*}
since $t-s$ and $(t-(i-1)\frac T {2^M})$ are smaller than $T/2^M$. 
Now, for $(t,s)\in \Delta_T$, we introduce the grid $(j-1) \frac{T}{2^M}\le s<j\frac{T}{2^M}\le\dots\le (i-1) \frac{T}{2^M}\le t<i\frac{T}{2^M}$ and use the previous inequality together with the subadditivity of $\R_+ \ni x\mapsto x^{2\gamma}$ to get the result. 

Last, we observe that if there exists $\varepsilon>0$ such that $\Gamma(t,s)\ge 0$ for $(t,s) \in \Delta$ with $s\ge t-\varepsilon$, then $\Gamma_M(t,t)\ge 0$ for $M$ such that $T/2^M<\varepsilon$, and therefore $\Gamma_M+\frac{1}
M$ satisfies the assumption and is positive on the diagonal. 
\end{proof}

\bibliographystyle{plainnat}
\bibliography{biblio_Volterra}

\end{document}